\documentclass[12pt,a4paper,reqno]{article}

\usepackage{amsmath}
\usepackage{amssymb}
\usepackage{amsthm}
\usepackage{xspace}
\usepackage[utf8]{inputenc}
\usepackage{color}
\usepackage{cite}

\newcommand{\ie}{{i.e.}\ }
\newcommand{\R}{\ensuremath{\mathbb{R}}}
\newcommand{\Rp}{\ensuremath{\mathbb{R}_+}}
\newcommand{\Z}{\ensuremath{\mathbb{Z}}}
\newcommand{\E}{\ensuremath{\mathbb{E}}}
\newcommand{\En}{\ensuremath{\mathcal{E}}}

\renewcommand{\d}{\,d}
\renewcommand{\epsilon}{\varepsilon}
\newcommand{\bdry}{\partial}

\DeclareMathOperator{\bal}{Bal}

\DeclareMathOperator{\gds}{GDS}
\DeclareMathOperator{\supp}{supp}

\newtheorem{theorem}{Theorem}[section]
\newtheorem{proposition}[theorem]{Proposition}

\newtheorem{lemma}[theorem]{Lemma}
\newtheorem{conjecture}[theorem]{Conjecture}
\theoremstyle{definition}
\newtheorem{definition}[theorem]{Definition}
\newtheorem{remark}[theorem]{Remark}

\makeatletter
\def\blfootnote{\xdef\@thefnmark{}\@footnotetext}
\makeatother

\title{Partial Balayage and a Generalization of the Divisible Sandpile Model }
\author{Joakim Roos \\[.1cm] \small{KTH Royal Institute of Technology}}
\date{}

\begin{document}

\maketitle

\vspace{.2cm}
\begin{abstract}
  In recent work by L.~Levine and Y.~Peres, it was observed that three
  models for particle aggregation on the lattice---the divisible
  sandpile, rotor-router aggregation, and internal diffusion limited
  agg\-regation---share a common scaling limit as the lattice spacing
  tends to zero, if they are started with the same initial mass
  configuration. It is straightforward to observe that this scaling
  limit is precisely the same as the potential-theoretic operation of
  taking the partial balayage of this initial mass configuration to
  the Lebesgue measure. However, from the theory of the partial
  balayage operation it is clear that one may take the partial
  balayage of a mass configuration to a more general measure than the
  Lebesgue measure, which one cannot do for the three aggregation
  models described by Levine and Peres. In this paper we therefore
  generalize one of these models, the divisible sandpile model, in
  mainly a bounded setting, and show that a natural scaling limit of
  this generalization is given by a general partial balayage
  operation.
\end{abstract}
\vspace{.2cm}

\blfootnote{\textit{2010 Mathematics Subject Classification.} Primary 31C20, Secondary 35R35.}
\blfootnote{\textit{Key words and phrases.} Divisible sandpile, partial balayage, obstacle problem.}

\section{Introduction}
In this section we review the results from L.~Levine and
Y.~Peres~\cite{levine-07,levine-peres-09} regarding the divisible
sandpile model (DS) and how its scaling limit is related to so-called
partial balayage to unit density, $\bal(\cdot, 1)$. Throughout, the
dimension $d$ will be assumed to satisfy $d \geq 2$.

\subsection{Preliminaries and main result}
\label{sec:prel-main-result}
Let $\mu: \R^d \to \Rp$ be a bounded and almost everywhere continuous
function, with the property that
$\left\{ x \in \R^d: \mu(x) \geq 1 \right\}$ is the closure of some
open bounded set $\Omega$. Given a decreasing sequence
$\{\xi_n\}_{n=1}^\infty$ of positive real numbers with limit zero as
$n \to \infty$ we define the \emph{discretized mass configuration}
$\mu_n: \xi_n \Z^d \to \Rp$ on the scaled lattice $\xi_n \Z^d$ by
\begin{align*}
  \mu_n(x) := \frac{1}{\xi_n^d} \int_{x^\square} \mu(y) \d y,
\end{align*}
where the symbol $x^\square$ denotes the closed cube in $\R^d$ of side
length $\xi_n$ and midpoint $x$, \ie the set
\begin{align*}
  x^\square := x + \left[ - \frac{\xi_n}{2}, \frac{\xi_n}{2}
  \right]^d.
\end{align*}
Since the volume of any such cube is
$\xi_n^d$ we see from the above that the value of a discretization
$\mu_n$ of $\mu$ at a point $x \in
\xi_n$ is nothing but the mean value of $\mu$ in the set
$x^\square$. We will also employ the notation that
$x^{::}$ is the closest lattice point to $x \in
\R^d$ (\ie if the lattice in question is $\xi \Z^d$, then $x^{::} = (x
+ (\xi/2, \xi/2]^d) \cap (\xi \Z^d)$). Moreover, if
$f$ is a function on $\R^d$ then
$f^{::}$ is defined as the restriction of
$f$ to the underlying lattice (determined by the context), and,
similarly, if $g$ is a lattice function on some lattice $\xi
\Z^d$, then $g^\square$ is the extension of
$g$ as a step function to $\R^d$ defined by $g^\square(x) :=
g(x^{::})$.

We need to say a few words about convergence of sequences
of sets relative to our sequence $\{\xi_n\}_{n=1}^\infty$ of
decreasing lattice constants: a sequence of sets
$\{A_n\}_{n=1}^\infty$, where $A_n \subset \xi_n \Z^d$, is said to
converge to a set $D \subset \R^d$ if there for any given $\epsilon >
0$ exists some integer $N$ such that we for all $n > N$ have
\begin{align*}
  D_\epsilon \cap \xi_n \Z^d \subset A_n \subset D^\epsilon,
\end{align*}
where $D_\epsilon$ and $D^\epsilon$ are subsets of $\R^d$,
the \emph{inner} and \emph{outer $\epsilon$-neighbourhoods of
  $D$}, respectively, defined by
\begin{align*} 
  D_\epsilon := \{x \in D : B(x, \epsilon) \subset D\}
\end{align*}
and
\begin{align*}
  D^\epsilon := \{x \in \R^d : B(x, \epsilon) \cap D \neq
  \emptyset\},
\end{align*}
so that $D_\epsilon \subset D \subset D^\epsilon$; here $B(a, \rho)$
is the open ball in $\R^d$ of radius $\rho > 0$ centred at $a \in
\R^d$.

Having treated the necessary technicalities the divisible sandpile
model on $\xi \Z^d$ for some lattice constant $\xi > 0$ is now defined
as follows: given a function $\mu: \xi \Z^d \to \Rp$, to be
interpreted as our initial mass configuration, we pick to begin with
any site $x \in \xi \Z^d$ for which $M := \mu(x) > 1$---we can think
of $\mu(x)$ to be the mass or number of (sand-)particles at $x$
(ignoring the fact that we very much allow for non-integral number of
particles), and hence that the site $x$ is chosen in such a way that
it has more than one particle. We now \emph{topple} the site $x$, by
which we mean that we leave a unit mass at $x$, and distribute the
remaining mass of $M - 1$ uniformly amongst the $2d$ neighbours $y$ of
$x$; for sake of simplicity we will write $y \sim x$ if $y$ is a
neighbour to $x$. In essence, we alter the mass configuration $\mu$ by
replacing $\mu(x)$ with $1$, and $\mu(y)$ with $\mu(y) + (M - 1)/2d$
for each $y \sim x$ to obtain a new mass configuration
$\mu': \xi \Z^d \to \Rp$. We now do the previous steps again starting
from $\mu'$ instead of $\mu$, and continue repeating this process over
and over again until we reach (in the limit) a final mass
configuration $\nu : \xi \Z^d \to \Rp$ which satisfies
$0 \leq \nu \leq 1$ everywhere. (That there even exists such a final
mass configuration $\nu$, not to mention the fact that this
configuration actually also is independent of the particular choice of
toppling sequence used, is highly non-trivial, but true under our
assumptions on $\mu$.) This process is what we call the
\emph{(standard) divisible sandpile}, and we call the final mass
configuration $\nu$ the \emph{(standard) divisible sandpile
  configuration of $\mu$ (on $\xi \Z^d$)}.

A highly important function $u$ called the \emph{odometer function}
can be defined for the divisible sandpile model: if $\xi \Z^d$ is the
lattice in question then $u$ is the function defined by letting $u(x)$
be $\xi^2$ times the total mass emitted from a lattice point
$x \in \xi \Z^d$ during the entire divisible sandpile process. Here
the factor $\xi^2$ is to ensure the proper limiting behaviour when we
later let $\xi \to 0$. If we study the algorithm for the divisible
sandpile model in detail it becomes clear that any site
$x \in \xi \Z^d$ will, in the end, have emitted a total mass of
$u(x) / 2d \xi^2$ to each of its $2d$ neighbours. But this reasoning
also applies to the neighbouring sites of $x$, hence each neighbour
$y$ will in total have sent mass of size $u(y) / 2d \xi^2$ to $x$. It
follows that the net \emph{increase} in mass at the site $x$ will be
the difference between the total mass received and the total mass
emitted, \ie precisely
\begin{align*}
  \sum_{y \sim x} \frac{u(y)}{2d\xi^2} - \sum_{y \sim x}
  \frac{u(x)}{2d\xi^2} = \sum_{y \sim x} \frac{u(y) - u(x)}{2d\xi^2} =
  \Delta u(x),
\end{align*}
where $\Delta$ is the (for our purposes suitably renormalized)
discrete Laplace operator. But this is only one way of calculating the
net increase of mass at $x$: with $\mu$ the initial mass
configuration and $\nu$ the mass configuration we end up with after
the aggregation is completed as above, we evidently have
\begin{align}
  \label{eq:Lapl-u-diff-mass-confs}
  \Delta u(x) = \nu(x) - \mu(x).
\end{align}
The main goal of our study is to calculate the resulting set of fully
occupied sites for the resulting divisible sandpile configuration, and
for this we observe that the odometer function $u$ can in fact be used
to determine this set completely. The set of such fully occupied sites
is of course the set
\begin{align*}
  D := \{x \in \xi \Z^d: \nu(x) = 1\}.
\end{align*}
If we consider any such $x \in D$ we must either have that no toppling
occurred at $x$ at any stage during the course of the divisible
sandpile algorithm, or that the site $x$ did topple at least once. If
$x$ did not topple, then no mass has left $x$, so $x$ must either have
had mass one during the entire course of the sandpile algorithm---if
so then $x$ must belong to the set $\{\mu \geq 1\}$---or must have
received mass from one from its neighbouring points, \ie must have a
neighbour that \emph{did} topple. On the other hand, if it in fact did
perform a toppling at some stage during the course of the algorithm,
then $u(x) > 0$. From these considerations we can conclude that, up to
possibly a (in some sense negligible) set of boundary points, the set
$D$ of fully occupied sites is essentially
\begin{align*}
  \{\mu \geq 1\} \cup \{u > 0\}.
\end{align*}
With this in mind, it is clear that we gain much information about the
set $D$ by finding the odometer function $u$, and the approach we will
take is to find $u$ as the solution to the equation
\eqref{eq:Lapl-u-diff-mass-confs}. Since $\nu$ will, by construction,
always satisfy $\nu \leq 1$, it is suitable to find a function
$\gamma$ that satisfies $\Delta \gamma(x) = 1 - \mu(x)$, since if we
then study the function $s' := \gamma - u$ we see that
\begin{align*}
  \Delta s'(x) &= \Delta(\gamma - u)(x) = \Delta \gamma(x) -
  \Delta u(x) \\
  &= 1 - \mu(x) - \nu(x) + \mu(x) = 1 - \nu(x) \geq 0
\end{align*}
holds everywhere, \ie $s'$ is a subharmonic function on $\xi
\Z^d$. We note that $s'$ satisfies $s' \leq \gamma$, since $u \geq 0$
by definition. Moreover, if $f$ is any other subharmonic function on
$\xi \Z^d$ satisfying $f \leq \gamma$, then
\begin{align*}
  \Delta(f - \gamma + u)(x) &= \Delta f(x) - 1 + \mu(x) + \nu(x) - \mu(x) \\
  &= \Delta f(x) - 1 + \nu(x) = \Delta f(x) \geq 0
\end{align*}
if $x \in D = \{\nu = 1\}$, and for $x$ outside $D$ we have $u(x) =
0$, hence
\begin{align*}
  f(x) - \gamma(x) + u(x) = f(x) - \gamma(x) \leq 0
\end{align*}
there. It follows that $f - \gamma + u$ is a nonpositive function
everywhere, \ie that $f \leq \gamma - u = s'$ on the whole of $\xi
\Z^d$. Thus, if we let $s$ be the subharmonic function defined by
\begin{align}
  \label{eq:s-def-on-lattice}
  s(x) := \sup \{f(x) : f \text{ is subharmonic in } \xi \Z^d
  \text{ and } f \leq \gamma\}
\end{align}
it follows both that $s \leq \gamma - u = s'$, but also $s \geq \gamma
- u$, since $s' = \gamma - u$ is a competing function in the set
defining $s$ in \eqref{eq:s-def-on-lattice}. We can conclude that we
in fact have
\begin{align*}
  u = \gamma - s
\end{align*}
where $s$ is given by \eqref{eq:s-def-on-lattice}.

We have converted the problem of finding the odometer function $u$, in
particular finding the set $\{ u > 0 \} = \{\gamma > s\}$, into
solving the \emph{obstacle problem} \eqref{eq:s-def-on-lattice}, a
problem that has a natural generalization to the continuous
setting. Therefore, given some initial mass configuration $\mu$ on
$\R^d$ (with appropriate assumptions on $\mu$ to ensure existence, and
so on) we define the obstacle $\gamma_c: \R^d \to \R$ by
\begin{align*}
  \gamma_c(x) := -|x|^2 - N * \mu(x)
\end{align*}
where $N(x)$ is the Newton kernel on $\R^d$, proportional to $\log
|x|^{-1}$ in two dimensions and to $|x|^{2-d}$ for $d \geq 3$, such
that $\Delta \gamma_c = \mu - 1$. As in
\eqref{eq:s-def-on-lattice}, we then define
\begin{align}
  \label{eq:s-def-on-cont}
  s_c(x) := \sup \{f(x) : f \in \mathcal{CS}(\R^d)
  \text{ and } f \leq \gamma_c\},
\end{align}
where $\mathcal{CS}(S)$ denotes the set of functions continuous and
subharmonic on some open set $S$. Assuming we can find a solution
$s_c$ to \eqref{eq:s-def-on-cont}, it can be seen that the set
\begin{align}
  \label{eq:def-set-D-from-obst-prob}
  D := \{ x \in \R^d : \gamma(x) > s_c(x)\}
\end{align}
will be the natural limit set, in the sense discussed above, to the
sequence of sets $\{u_n > 0\}$ where $u_n$ is the $n$th
odometer function for the divisible sandpile model for a sequence of
decreasing positive lattice constants $\{\xi_n\}_{n=1}^\infty$
converging to zero.

We are now ready to state one of the main results from Levine's thesis
\cite{levine-07}:
\begin{theorem}
  \label{thm:levine-main-thm}
  Let $\{\xi_n\}_{n=1}^\infty$ be a decreasing sequence of positive
  real numbers converging to zero, and let $\{\mu_n\}_{n=1}^\infty$ be
  a discretized mass configuration based on the sequence
  $\{\xi_n\}_{n=1}^\infty$ for some mass configuration $\mu: \R^d
  \to \Rp$ as above, with $\Omega$ the open bounded set satisfying
  $\overline{\Omega} = \{\mu \geq 1\}$. Let $D_n$ be the domain of
  occupied sites from the standard divisible sandpile in the lattice
  $\xi_n \Z^d$ started from source density $\mu_n$. Then, as $n
  \to \infty$,
  \begin{align*}
    D_n \to D \cup \Omega,
  \end{align*}
  where $D$ is the set given by
  \eqref{eq:def-set-D-from-obst-prob}.
\end{theorem}
We will later on observe that the obstacle problem in
\eqref{eq:s-def-on-cont} is essentially the same obstacle problem as
that occurring in the definition of the partial balayage operation
$\bal(\cdot, m)$ of a mass configuration to the Lebesgue measure
$m$---\ie to density one, if we think of $m$ as a distribution---and
so the limiting set in the above theorem is precisely
\begin{gather*}
  D \cup \Omega = \supp \bal(\mu, m).
\end{gather*}
One consequence of the above is that if we let $\nu_n$ be the result
of the standard divisible sandpile on $\xi_n \Z^d$ started from
density $\mu_n$, but choose to interpret this resulting mass
configuration as a measure on $\R^d$, \ie with some abuse of notation
we let
\begin{align*}
  \nu_n = \xi_n^d \sum_{x \in \xi_n \Z^d} \nu_n(x) \cdot \delta_{x},
\end{align*}
where $\delta_x$ is the Dirac point mass measure at $x$, then
$\nu_n \to \bal(\mu, m)$ in the sense of distributions as
$n \to \infty$. This weak form of convergence is the approach we will
take in the remainder of the paper.

\section{Partial balayage}
\label{sec:partial-balayage}
In this paper we are going to refer to two different variants of
partial balayage: first a bounded version with Dirichlet boundary
conditions, which we are going to relate to a bounded version of the
generalized divisible sandpile algorithm, and also an unrestricted
version when the dimension $d = 2$, which we in turn relate to a the
possible limit of the generalized divisible sandpile in the setting
where the confining radius grows infinitely large.

\subsection{Bounded partial balayage}
\label{sec:partial-balayage-bdd}
The bounded version of partial balayage was developed by
B.~Gust\-afsson and M.~Sakai in~\cite{gustafsson-sakai-94}, which we
include here mainly for sake of completeness and for an overview of
the minor adjustments to the notation we use in this paper. For proofs
we refer to~\cite{gustafsson-sakai-94}, and for a good survey of
partial balayage in general, see for instance \cite{gustafsson-04}.

Before we continue, we need to say a few words about our notation. If
$\mu$ is a signed Radon measure on $\R^d$ with compact support, then
we denote by $U^\mu$ the Newtonian potential of $\mu$. For greater
compatibility with the analogous theory in the discrete setting, we
use the (somewhat non-standard) normalization of the potential such
that
\begin{align*}
  -\Delta U^\mu(x) = 2d \cdot \mu(x),
\end{align*}
which always holds in the sense of distributions (and pointwise
wherever $U^\mu$ is $C^2$). Here $\Delta$ is the usual Laplace
operator
\begin{align*}
  \Delta = \sum_{i=1}^d \frac{\partial^2}{\partial x_i^2},
\end{align*}
with the natural generalization in terms of distributions.

\begin{definition}
\label{def:partial-balayage-def-set}
  Let $\sigma = \sigma_+ - \sigma_-$ be a signed Radon measure on
  $\R^d$ with compact support, and let $R > 0$. Define the set
  \begin{gather*}
    \mathcal{F}^{\sigma,R} := \{V \in \mathcal{D}'(\R^d): V \leq
    U^\sigma \text{ in } \R^d, \Delta V \geq 0 \text{ in } B(0, R)\},
  \end{gather*}
  where $\mathcal{D}'(\R^d)$ is the set of distributions in $\R^d$.
\end{definition}
\begin{theorem}
\label{thm:partial-balayage-bdd-sol-fn-props}
  The set $\mathcal{F}^{\sigma,R}$ in
  Definition~\ref{def:partial-balayage-def-set} contains a largest
  element,
  $V^\sigma \equiv V^{\sigma,R} := \sup \mathcal{F}^{\sigma,R}$. This
  $V^\sigma$ satisfies the complementarity system
  \begin{align*}
      V^\sigma &\leq U^\sigma \text{ in } \R^d, \\
      \Delta V^\sigma &\geq 0 \text{ in } B(0, R), \\
      V^\sigma &= U^\sigma \text{ on } \R^d \setminus B(0, R), \\
      -\Delta V^\sigma &= 0 \text{ in } \omega(\sigma) := \{x \in B(0,R) : V^\sigma(x) < U^\sigma(x)\}.
  \end{align*}
  It follows from the above that $-\Delta V^\sigma$ is a signed Radon
  measure.
\end{theorem}
\begin{definition}
  \label{def:partial-balayage-bdd}
  The \emph{partial balayage} relative to the ball $B(0,R)$, where
  $R > 0$, of a signed Radon measure $\sigma = \sigma_+ - \sigma_-$
  with compact support is defined to be the signed Radon measure
  \begin{gather*}
    \bal_R(\sigma, 0) := -\frac{1}{2d}\Delta V^{\sigma, R},
  \end{gather*}
  where $V^{\sigma, R}$ is as in
  Theorem~\ref{thm:partial-balayage-bdd-sol-fn-props}.
\end{definition}

\begin{remark}
  \label{remark:mod-schwarz-potential}
  The \emph{modified Schwarz potential} of the above problem is the
  function $u \equiv u^{\sigma, R} := U^\sigma - V^{\sigma, R}$. In
  terms of the complementarity system in
  Theorem~\ref{thm:partial-balayage-bdd-sol-fn-props}, $u$ satisfies
  \begin{align*}
    u &\geq 0 \text{ in } \R^d, \\
    \Delta u &\geq 2d \cdot \sigma \text{ in } B(0, R), \\
    u &= 0 \text{ on } \R^d \setminus B(0, R), \\
    u &= 2d \cdot \sigma \text{ in } \omega(\sigma) := \{x \in B(0,R) : u(x) > 0\}.
  \end{align*}
  For the partial balayage measure in terms of $u$, we see from
  Definition~\ref{def:partial-balayage-bdd} that $\bal_R(\sigma, 0)$
  is given by
  \begin{gather*}
    \bal_R(\sigma, 0) = -\frac{1}{2d}\Delta (U^\sigma - u) = \sigma +
    \frac{1}{2d} \Delta u.
  \end{gather*}
\end{remark}

\begin{remark}
  \label{rem:def-of-bdd-bal-mu-to-lambda}
  In \cite{gustafsson-sakai-94}, and several other articles, the
  partial balayage operation is often discussed in terms of
  $\nu := \bal_R(\mu, \lambda)$, where $\mu$ and $\lambda$ are
  suitable (positive) measures. This resulting (also positive) measure
  $\nu$ then satisfies $\nu \leq \lambda$ in $B(0,R)$, and $U^\nu$
  being equal to $U^\mu$ (\ie $\nu$ and $\mu$ are ``graviequivalent'')
  outside of some \emph{a priori} unknown set $\Omega$. At least in
  the finite energy setting, this $\nu$ is the unique minimizer of the
  energy norm difference
  \begin{align*}
    I[\mu - \nu] = \int U^{\mu - \nu} \d (\mu - \nu)
  \end{align*}
  over all $\nu$ satisfying $\nu \leq \lambda$ in $B(0,R)$ (and, at
  least in two dimensions, with the extra condition that $\nu$ has the
  same total mass as $\mu$).

  In this paper, we will mostly focus on partial balayage measures of
  the form $\bal_R(\cdot, 0)$, as defined in
  Definition~\ref{def:partial-balayage-bdd}. At times when we need to
  refer to partial balayage measures of the form
  $\bal_R(\cdot, \lambda)$ instead, we utilize a well-known
  translational invariance property of partial balayage (see
  \cite{gustafsson-sakai-94}), in that, for suitable measures $\mu$,
  $\lambda$ and $\eta$ to ensure existence,
  \begin{align}
    \label{eq:pb-translational-invariance}
    \bal_R(\mu + \eta, \lambda + \eta) = \bal_R(\mu, \lambda) + \eta.
  \end{align}
  In other words, when appropriate we simply think of
  $\bal_R(\mu, \lambda)$ as the measure defined by
  \begin{align*}
    \bal_R(\mu, \lambda) = \bal_R(\mu - \lambda, 0) + \lambda.
  \end{align*}
\end{remark}

\subsection{Unrestricted partial balayage in the plane}
\label{sec:unrestr-part-balayage}
In the plane it is known that
Definition~\ref{def:partial-balayage-bdd}, under suitable assumptions
on the signed measure $\sigma$, can be generalized to allow for an
infinite confining radius. See~\cite{roos-15} for details, and for
recently developed connections between partial balayage measures and
equilibrium measures in weighted potential
theory~\cite{balogh-harnad-09,saff-totik-97}.
\begin{definition}
  \label{def:unbdd-partial-balayage-def-set}
  Let $\sigma = \sigma_+ - \sigma_-$ be a signed Radon measure on
  $\R^2$ with compact support. Define the set
  \begin{gather*}
    \mathcal{F}^{\sigma} := \{V \in \mathcal{D}'(\R^d): V \leq
    U^\sigma \text{ in } \R^2, \Delta V \geq 0 \text{ in } \R^2\}.
  \end{gather*}
\end{definition}
\begin{theorem}
  \label{thm:partial-balayage-unbdd-sol-fn-props}
  If $\sigma = \sigma_+ - \sigma_-$ is a signed Radon measure on
  $\R^2$ with compact support and negative total mass, with the
  property that $U^{\sigma_-}$ is a continuous function on $\R^2$,
  then $\mathcal{F}^\sigma$ is non-empty and contains its largest
  element, ${V^\sigma := \sup \mathcal{F}^\sigma}$. This $V^\sigma$
  satisfies the complementarity system
  \begin{align*}
    V^\sigma &\leq U^\sigma \text{ in } \R^2, \\
    \Delta V^\sigma &\geq 0 \text{ in } \R^2, \\
    V^\sigma &= U^\sigma \text{ in } \supp \Delta V^\sigma \subset
               \supp \sigma_-, \\
    -\Delta V^\sigma &= 0 \text{ in } \omega(\sigma) := \{x \in \R^2 :
                       V^\sigma(x) < U^\sigma(x)\}.
  \end{align*}
\end{theorem}
\begin{definition}
  \label{def:partial-balayage-unbdd}
  The \emph{(unrestricted) partial balayage} of a signed Radon measure
  $\sigma = \sigma_+ - \sigma_-$ with compact support, assumed to
  satisfy $U^{\sigma_-}$ continuous everywhere on $\R^2$, is defined
  to be the signed Radon measure
  \begin{gather*}
    \bal(\sigma, 0) := -\frac{1}{2d}\Delta V^{\sigma},
  \end{gather*}
  where $V^{\sigma}$ is as in
  Theorem~\ref{thm:partial-balayage-unbdd-sol-fn-props}.
\end{definition}

\section{Generalizing the divisible sandpile}
As mentioned earlier, the scaling limit of the standard divisible
sandpile obtained in L.~Levine's thesis~\cite{levine-07} is related to
taking partial balayage of a mass configuration to the Lebesgue
measure $m$, \ie $\bal_R(\mu, 1) \equiv \bal_R(\mu, m)$, where $R > 0$
is a large enough bounding radius.

However, as we saw in Section~\ref{sec:partial-balayage}, there is
mathematically no problem in calculating the partial balayage of a
mass configuration relative to a more general measure than the
Lebesgue measure, \ie instead calculating $\bal_R(\mu, \lambda)$,
where $\lambda$ is a measure that, in a sense, describes the maximal
density that will be allowed for the final mass configuration. It is
therefore a natural question to ask if the standard divisible sandpile
model in~\cite{levine-07} can be generalized to incorporate this
measure $\lambda$, in such a way that the corresponding scaling limit
of this modified particle aggregation model coincides with
$\bal_R(\mu, \lambda)$.

In this section we shall see that this is, indeed, possible. With the
translational invariance \eqref{eq:pb-translational-invariance} in
mind, we will, mainly for sake of simplicity in the formulation,
actually develop a generalized sandpile model that converges to
measures of the form $\bal_R(\cdot, 0)$ in the appropriate scaling
limit. If desired, this can then readily be reformulated into a
corresponding result in terms of $\bal_R(\mu, \lambda)$.

\subsection{Bounded divisible sandpile for signed mass configurations
  on a fixed lattice}
Let $\sigma: \xi \Z^d \to \R$ be a bounded function on the lattice
$\xi \Z^d$ for some lattice constant $\xi > 0$; this function
will be our generalization of the initial mass configuration. We
shall always assume that $\sigma$ has compact support
\begin{align*}
  \supp \sigma := (\supp \sigma_+) \cup (\supp \sigma_-),
\end{align*}
where $\sigma_+ = \max(\sigma, 0)$ and $\sigma_- = - \min(\sigma, 0)$;
a bounded lattice function of compact support will for sake of brevity
be called a \emph{generalized mass configuration}. We are only going
to be interested in \emph{admissible} generalized mass configurations,
by which we mean
\begin{align}
  \label{eq:gds-compatibility-req}
  \sum_{x \in \xi \Z^d} \sigma_-(x) \geq \sum_{x \in \xi\Z^d}
  \sigma_+(x).
\end{align}
Much like we in the standard divisible sandpile model ended up with a
mass configuration satisfying $\nu \leq 1$ everywhere, we will, in our
generalized divisible sandpile, in the end obtain a generalized mass
configuration $\nu$ satisfying $\nu \leq 0$ everywhere. Since we want
the \emph{total} mass of our mass configuration to remain the same
throughout this process, so that $\sum_{x \in \xi\Z^d} \nu(x) =
\sum_{x \in \xi\Z^d} \sigma(x)$, this explains requirement
\eqref{eq:gds-compatibility-req}, as we then have
\begin{gather*}
  \sum_{x \in \xi\Z^d} \sigma(x) = \sum_{x \in \Z^d}
  \underbrace{\nu(x)}_{\leq 0} \leq 0.
\end{gather*}
The main way we will generalize the divisible sandpile model is by
generalizing the toppling step described in
Section~\ref{sec:prel-main-result} for the standard divisible
sandpile. In the standard model, at every site $x$ where the mass $M$
exceeds one, we redefine our mass configuration locally around $x$,
leaving a mass of one at $x$ and spreading the remaining mass $M - 1$
equally amongst the nearest neighbours of $x$. We here essentially do
more or less the same, with the difference that we instead look for
sites where $\sigma$ is positive (\ie violating the desired property
of the mass configuration being $\leq 0$ everywhere). Thus, for every
site $x$ in our lattice where we have $\sigma_+(x) > 0$ we modify our
mass configuration around $x$, leaving no mass at all at $x$ (so that
$\nu \leq 0$ at least is satisfied at $x$ for our new mass
configuration $\nu$), and relocate the remaining mass $\sigma_+(x)$
equally amongst the $2d$ neighbouring sites of $x$.

To formalize the above we do the following: consider $x \in
\xi\Z^d$ arbitrary but for the moment fixed, let $\eta: \xi\Z^d
\to \R$ be some generalized mass configuration and define toppling of
$\eta$ at the site $x$ to be the mass configuration $T_x \eta$ defined
by
\begin{align}
  \label{eq:def-gen-toppling}
  T_x \eta(y) := \eta(y) + \eta_+(x) \xi^2 \cdot \Delta \delta_x(y),
\end{align}
where $\delta_x$ is the (discrete) delta function at $x$, and $\Delta$
is the (for our purposes suitably normalized) discrete Laplace
operator defined by
\begin{align}
  \label{eq:def-discr-Laplacian}
  \Delta f(y) = \frac{1}{2d \xi^2} \sum_{y' \sim y}
  \left( f(y') - f(y) \right),
\end{align}
where $y' \sim y$ means $y' \in \xi\Z^d$ is one of the $2d$
neighbouring points of distance $\xi$ from $y$ in $\xi\Z^d$. If $x$
happens to be a lattice point for which $\eta(x) \leq 0$ holds, then
clearly $\eta_+(x) = 0$, hence $T_x \eta(y) = \eta(y)$ holds for
every $y \in \xi\Z^d$, as desired. If we on the other hand happen to
have $\eta(x) = \eta_+(x) > 0$ then we get a contribution from the
second term in \eqref{eq:def-gen-toppling} and need to calculate
$\Delta \delta_x(y)$ to determine what $T_x \eta(y)$ is. From
\eqref{eq:def-discr-Laplacian} we obtain
\begin{align*}
  \Delta \delta_x(y) = \frac{1}{2d\xi^2} \sum_{y' \sim y} \left(
    \delta_x(y') - \delta_x(y) \right),
\end{align*}
and see that this function obtains different values depending on how
close $y$ is to $x$. If $y = x$, then $\delta_x(y) = 1$ and
$\delta_x(y') = 0$ for every $y' \sim y = x$, from which it follows
that
\begin{align*}
  \Delta \delta_x(x) = \frac{1}{2d\xi^2} \sum_{y' \sim y} \left(
    0 - 1 \right) = - \frac{1}{2d\xi^2} \sum_{y' \sim y} 1 = -
  \frac{2d}{2d\xi^2} = -\frac{1}{\xi^2}.
\end{align*}
If $y$ instead is a neighbouring point of $x$, then $x$ is a
neighbouring point of $y$ (naturally), so $\delta_x(y')$ will be zero
for every $y' \sim y$ except for when $y' = x$. Clearly we then also
have $\delta_x(y) = 0$ as $y \neq x$, so we in this case instead
obtain
\begin{align*}
  \Delta \delta_x(y) = \frac{1}{2d\xi^2} \sum_{y' \sim y} \left(
    \delta_x(y') - 0 \right) = \frac{1}{2d\xi^2}.
\end{align*}
Finally, if $y$ is neither equal to $x$ nor a neighbouring point of
$x$, then $\delta_x(y')$ is zero for every $y' \sim y$ and evidently
also $\delta_x(y) = 0$, yielding $\Delta \delta_x(y) = 0$. We
summarize these cases into
\begin{align*}
  \Delta \delta_x(y) = \left\{
      \begin{array}{l l}
        \displaystyle -\frac{1}{\xi^2} & \text{if } y = x, \\[.4cm]
        \displaystyle \frac{1}{2d\xi^2} & \text{if } y \sim x, \\[.4cm]
        0 & \text{otherwise.}
      \end{array}
    \right.
\end{align*}
This yields that we obtain
\begin{align*}
  T_x \eta(y) &= \left\{
      \begin{array}{l l}
        \displaystyle \eta(x) + \eta_+(x) \xi^2 \cdot (-\frac{1}{\xi^2}) & \text{if } y = x, \\[.4cm]
        \displaystyle \eta(y) + \eta_+(x) \xi^2 \cdot \frac{1}{2d \xi^2} & \text{if } y \sim x, \\[.4cm]
        \eta(y) & \text{otherwise;}
      \end{array}
    \right. \\
    &= \left\{
      \begin{array}{l l}
        -\eta_-(x) & \text{if } y = x, \\[.2cm]
        \displaystyle \eta(y) + \frac{\eta_+(x)}{2d} & \text{if } y \sim x, \\[.4cm]
        \eta(y) & \text{otherwise.}
      \end{array}
    \right.
\end{align*}
We see that this way of defining the toppling agrees precisely with
how we want to modify the mass configuration if $x$ is a site where
the mass configuration has a violating positive mass.

Naturally, the site $x$ need not be the only site in $\xi\Z^d$ where
the initial mass configuration possibly is in violation of the desired
nonpositivity, and we also note that as we perform the above toppling
at $x$ we could in fact turn some of the neighbouring points of $x$
into violating points if we add too much mass to these points during
the toppling process. To ensure that we in the end obtain a mass
configuration $\nu$ which satisfies $\nu \leq 0$ \emph{everywhere},
and not only at specific points, we therefore need to do this toppling
procedure over all violating points and repeat when necessary. To
avoid problems with mass possibly escaping to infinity, we will in
this section treat a \emph{bounded} generalization of the divisible
sandpile, \ie fix some $R > 0$ and restrict our study for the moment
to the set $\hat{B}_R := B(0, R) \cap \xi\Z^d$, where
$B(a, r) \subset \R^d$ is the open ball centred at $a \in \R^d$ of
radius $r > 0$; we choose $R$ large enough so that $\hat{B}_R$
contains the support of our initial mass configuration. Now fix a
sequence $x_1, x_2, \ldots$ of points of $\hat{B}_R$ with the property
that if $x \in \hat{B}_R$ is arbitrary, then there are infinitely many
points in the sequence $x_1, x_2, \ldots$ for which $x_k = x$; we call
such a sequence an \emph{infinitely covering sequence} (of
$\hat{B}_R$). For $k \geq 1$ we define the mass configuration
$\sigma_k^R \equiv \sigma_k$ to be the mass configuration obtained
from $\sigma$ after successive toppling of the sites
$x_1, x_2, \ldots, x_k$, \ie we let
\begin{align*}
  \sigma_k(y) := T_{x_k} T_{x_{k-1}} \ldots T_{x_2} T_{x_1} \sigma
  (y).
\end{align*}
Also, for each $k \geq 1$ we define the \emph{$k$th odometer function}
$u_k : \xi\Z^d \to \R_+$ to be $\xi^2$ times the total mass emitted
from the site $x$ after toppling the sites $x_1, x_2, \ldots,
x_k$. These odometer functions are, as already seen in the
introduction, highly useful when studying what happens to the mass
configuration as $k$ tends to infinity.

For any subset $S \subset \xi\Z^d$ of the lattice, we will by
$\partial S$ denote the outer boundary of $S$, defined by
\begin{align*}
  \partial S := \{y \notin S: \text{ there exists } y' \in S \text{
    with } y \sim y'\};
\end{align*}
note that we by definition always have
$S \cap \partial S = \emptyset$. Our first main result is the
following:
\begin{proposition}
  \label{prop:bounded-gds}
  Let $\sigma : \xi\Z^d \to \R$ be a generalized mass
  configuration, let $R > 0$ be such that $\supp \sigma \subset
  \hat{B}_R := B(0, R) \cap \xi\Z^d$ and let $x_1, x_2, x_3,
  \ldots$ be an infinitely covering sequence of $\hat{B}_R$. For each
  $k \geq 1$ let $\sigma_k$ be the generalized mass configuration
  obtained from $\sigma$ after toppling the $k$ points $x_1, \ldots,
  x_k$, and let $u_k$ be the corresponding odometer function.

  Then there exists a generalized mass configuration $\nu$ on
  $\xi\Z^d$ and a function $u : \xi\Z^d \to \R_+$ such that
  $\sigma_k(x) \to \nu(x)$ and $u_k(x) \nearrow u(x)$ for every $x \in
  \xi\Z^d$ as $k \to \infty$. Moreover, $\nu = \nu_+ - \nu_-$ has
  the structure $\supp \nu_+ \subset \bdry \hat{B}_R$ and $\supp \nu_-
  \subset \supp \sigma_-$, so that $\nu \geq 0$ on $\bdry \hat{B}_R$
  and $\nu \leq 0$ on $\hat{B}_R$.
\end{proposition}
\textbf{Note:} The proof of the above proposition is essentially
identical to the proof of the analogous statement for the standard
divisible sandpile, as given in Lemma~3.1 in~\cite{levine-peres-09},
with only minor adjustments for change in notation and the restriction
that our infinitely covering sequence now is a subset of $\hat{B}_R$
instead of $\xi\Z^d$ as in~\cite{levine-peres-09}; we include it
here for completeness.
\begin{proof}
  It is evident from the definition of the toppling procedure that the
  $k$th mass configuration $\sigma_k$ can only be nonzero in
  $\hat{B}_R$ (the set covered by the sites at which we perform
  toppling) and possibly also on the boundary of $\hat{B}_R$, so for
  every $k$ we have $\sigma_k(x) = 0$ if $|x| \geq R + 2$. We define
  the $k$th quadratic weight $Q_k$ through
  \begin{gather}
    \label{eq:def-quadr-weight}
    Q_k := \sum_{x \in \xi\Z^d} \sigma_k(x) |x|^2.
  \end{gather}
  (Here $\sigma_0 = \sigma$.) On one hand, this immediately yields
  \begin{gather}
    Q_k = \sum_{x \in \xi\Z^d} ((\sigma_k)_+(x) - (\sigma_k)_-(x)) |x|^2
    \leq \sum_{x \in \xi\Z^d} (\sigma_k)_+(x) |x|^2.
  \end{gather}
  We now claim that for every $k \geq 1$ we have
  \begin{gather}
    \label{eq:sigma-k-plus-iter-ineq}
    \sum_{x \in \xi\Z^d} (\sigma_k)_+(x) \leq \sum_{x \in \xi\Z^d}
    (\sigma_{k-1})_+(x),
  \end{gather}
  which then, by iteration and the fact that $\sigma_k(x) = 0$ for all
  $|x| \geq R + 2$, leads to the inequality
  \begin{gather}
    Q_k \leq (R + 2)^2 M_+ \text{ for all } k \geq 0,
  \end{gather}
  where $M_+$ is the total mass of the non-negative part of the
  initial mass configuration $\sigma$:
  \begin{gather*}
    M_+ := \sum_{x \in \xi\Z^d} \sigma_+(x).
  \end{gather*}
  To prove \eqref{eq:sigma-k-plus-iter-ineq}, we first observe that if
  $(\sigma_{k-1})_+(x_k) = 0$ there is nothing to prove, since
  $\sigma_k(x) = T_{x_k}\sigma_{k-1}(x) = \sigma_{k-1}(x)$ then holds
  for every $x$. For now we therefore assume
  $(\sigma_{k-1})_+(x_k) > 0$. This implies that
  $(\sigma_k)_+(x_k) = 0$, $(\sigma_k)_+(x) = (\sigma_{k-1})_+(x)$ for
  all $x \neq x_k$ with $x \not \sim x_k$, and for every $x \sim x_k$
  the inequality
  \begin{gather}
    (\sigma_k)_+(x) \leq (\sigma_{k-1})_+(x) +
    \frac{(\sigma_{k-1})_+(x_k)}{2d}
  \end{gather}
  holds. The left hand side of \eqref{eq:sigma-k-plus-iter-ineq} then
  becomes
  \begin{align}
    \sum_{x \in \xi\Z^d} &(\sigma_k)_+(x) = \sum_{x \sim x_k} (\sigma_k)_+(x) + \sum_{\substack{x \neq x_k, \\ x \not \sim x_k}} (\sigma_k)_+(x) \nonumber \\
    &\leq \sum_{x \sim x_k} \left( (\sigma_{k-1})_+(x) +
      \frac{(\sigma_{k-1})_+(x_k)}{2d} \right) + \sum_{\substack{x \neq x_k, \\ x \not \sim x_k}} (\sigma_{k-1})_+(x) \nonumber \\
    &= (\sigma_{k-1})_+(x_k) + \sum_{x \neq x_k} (\sigma_{k-1})_+(x) =
    \sum_{x \in \xi\Z^d} (\sigma_{k-1})_+(x),
  \end{align}
  as desired. Since the total mass of $\sigma_k$ is equal to the total
  mass of $\sigma_{k-1}$ by construction, inequality
  \eqref{eq:sigma-k-plus-iter-ineq} immediately implies that we also
  have
  \begin{gather}
    \label{eq:sigma-k-minus-iter-ineq}
    \sum_{x \in \xi\Z^d} (\sigma_k)_-(x) \leq \sum_{x \in \xi\Z^d}
    (\sigma_{k-1})_-(x)
  \end{gather}
  for each $k \geq 1$, which, in a similar manner, in turn implies the
  lower bound
  \begin{gather}
    Q_k \geq - (R + 2)^2 M_- \text{ for all } k \geq 0,
  \end{gather}
  where $M_-$ is the total mass of the non-positive part of the
  initial mass configuration $\sigma$:
  \begin{gather*}
    M_- := \sum_{x \in \xi\Z^d} \sigma_-(x).
  \end{gather*}
  We have thus established the following bounds on $Q_k$ for each
  $k \geq 0$:
  \begin{gather}
    - (R + 2)^2 M_- \leq Q_k \leq (R + 2)^2 M_+.
  \end{gather}
  From \eqref{eq:def-quadr-weight} it follows for $k \geq 1$ that
  \begin{gather}
    \label{eq:quadr-weight-calc-1}
    Q_k - Q_{k-1} = \sum_{x \in \xi\Z^d} (\sigma_k(x) -
    \sigma_{k-1}(x))|x|^2,
  \end{gather}
  and from the definition of $\sigma_k$ as
  $\sigma_k = T_{x_k} \sigma_{k-1}$ one obtains slightly different but
  related results depending on if the $x \in \xi\Z^d$ is equal to the
  toppling point $x_k$, is merely adjacent to $x_k$, or neither of
  these: for $x = x_k$ a trivial calculation shows that
  $\sigma_k(x_k) - \sigma_{k-1}(x_k) = -(\sigma_{k-1})_+(x_k)$, if
  instead $x \sim x_k$ then
  $\sigma_k(x) - \sigma_{k-1}(x) = \frac{1}{2d}(\sigma_{k-1})_+(x_k)$
  holds, and if $x$ is neither equal to nor adjacent to $x_k$ then
  $\sigma_k(x) = \sigma_{k-1}(x)$. Inserting these results into
  \eqref{eq:quadr-weight-calc-1} yields
  \begin{align*}
    Q_k - Q_{k-1} &= -(\sigma_{k-1})_+(x_k) |x_k|^2 +
    \frac{1}{2d}(\sigma_{k-1})_+(x_k) \sum_{x \sim x_k}
    |x|^2 \\
    &= (\sigma_{k-1})_+(x_k) \cdot \underbrace{\frac{1}{2d} \sum_{x
        \sim
        x_k}\left(|x|^2 - |x_k|^2 \right)}_{=\xi^2 \cdot (\Delta |x|^2)(x_k)=\xi^2} \\
    &= \xi^2 \cdot (\sigma_{k-1})_+(x_k).
  \end{align*}
  This in turn implies that
  \begin{gather}
    \label{eq:quadr-weight-calc-2}
    Q_k = Q_0 + \xi^2 \cdot \sum_{j=1}^k (\sigma_{j-1})_+(x_j).
  \end{gather}
  Now consider the $k$th odometer function $u_k$: the value of
  $u_k(x)$ is defined as $\xi^2$ times the total mass emitted from $x$
  during the $k$ first applications of the toppling procedure,
  therefore we can write the value of $u_k$ at $x$ as
  \begin{gather*}
    u_k(x) = \xi^2 \sum_{1 \leq j \leq k: x_j = x} (\sigma_{j-1})_+(x)
  \end{gather*}
  If we now sum $u_k(x)$ over all $x \in \xi\Z^d$, keeping in mind that
  $u_k(x)$ will be zero for every $x$ outside $\hat{B}_R$ and that our
  sequence $x_1, x_2, \ldots$ is an infinitely covering sequence of
  $\hat{B}_R$, then we obtain
  \begin{gather}
    \label{eq:sum-of-uk-is-sum-of-toppled-mass}
    \sum_{x \in \xi\Z^d} u_k(x) = \xi^2 \sum_{x \in \xi\Z^d} \sum_{1
      \leq j \leq k:x_j = x} (\sigma_{j-1})_+(x) = \xi^2 \cdot \sum_{j=1}^k
    (\sigma_{j-1})_+(x_j).
  \end{gather}
  Combining this last result with \eqref{eq:quadr-weight-calc-2} and
  our previously established bounds $Q_k \leq (R+2)^2 M_+$ and
  $-Q_0 \leq (R+2)^2 M_-$, we get
  \begin{gather}
    \label{eq:quadr-weight-calc-3}
    \sum_{x \in \xi\Z^d} u_k(x) \leq (R+2)^2 M,
  \end{gather}
  where $M := M_+ + M_-$. As the right side of
  \eqref{eq:quadr-weight-calc-3} is independent of $k$, and $u_k(x)$
  clearly is an increasing function of $k$ for each fixed $x \in
  \xi\Z^d$, it follows that for any fixed $x \in \xi\Z^d$ the
  sequence $\{u_k(x)\}_{k=1}^\infty$ is increasing and bounded from
  above, hence convergent. We define the \emph{odometer function} $u$
  to be this limit: for any $x \in \Z^d$ let
  \begin{gather}
    \label{eq:gds-def-odometer-function}
    u(x) := \lim_{k \to \infty} u_k(x).
  \end{gather}
  Now, if $y \sim x$ then it is clear from how we defined the toppling
  that $x$ after $k$ toppling steps has received a contribution of
  mass of size $\frac{1}{2d\xi^2}u_k(y)$ from $y$. Since this holds
  for each neighbouring point of $x$, it is clear that $x$ in total has
  received a mass of size $\frac{1}{2d\xi^2} \sum_{y \sim x} u_k(y)$
  after the $k$th toppling step. But during these steps we may also
  have performed toppling at $x$ itself, so to calculate the net
  difference in mass at $x$ at the $k$th step from our initial mass
  configuration at $x$ we need to subtract the mass emitted from $x$
  up to this point, \ie $u_k(x) / \xi^2$, from the total mass
  received. Hence we see that
  \begin{gather}
    \label{eq:gds-kth-mass-config-and-odom}
    \sigma_k(x) = \sigma(x) + \frac{1}{2d\xi^2} \sum_{y \sim x} (u_k(y) -
    u_k(x)) = \sigma(x) + \Delta u_k(x).
  \end{gather}
  However, we just showed that $u_k$ had a well-defined limit as $k$
  tends to infinity, and so relation
  \eqref{eq:gds-kth-mass-config-and-odom} shows that also $\sigma_k$
  has a limit, namely
  \begin{gather}
    \label{eq:gds-limiting-mass-config}
    \nu := \sigma + \Delta u.
  \end{gather}
  Finally, the proposed structure of $\nu = \nu_+ + \nu_-$ with $\supp
  \nu_+ \subset \partial \hat{B}_R$ and $\supp \nu_- \subset \supp
  \sigma_-$ is now evident: for any $x \in \hat{B}_R$ we have for
  infinitely many values of $k$ that $\sigma_k(x) \leq 0$ holds true
  (namely whenever we just toppled at $x$), hence the same inequality
  must hold for the limiting mass configuration that we now know
  exists, \ie $\nu(x) \leq 0$ for all $x \in \hat{B}_R$. Iteration of
  the estimate $(\sigma_k)_-(x) \leq (\sigma_{k-1})_-(x)$ for any $x
  \in \xi\Z^d$ and $k \geq 1$ implies that ${(\sigma_k)_-(x) \leq
  \sigma_-(x)}$, which in the limit $k \to \infty$ becomes $\nu_-(x)
  \leq \sigma_-(x)$, establishing $\supp \nu_- \subset \supp
  \sigma_-$. Finally, the fact that we only perform toppling in the
  set $\hat{B}_R$ implies that $\nu$ in principle only can be non-zero
  on the set $\hat{B}_R \cup \{x: x \sim y \text{ where } y \in
  \hat{B}_R\} = \hat{B}_R \cup \partial \hat{B}_R$. However, since we
  already know that $\nu$ is non-positive on $\hat{B}_R$, it follows,
  as desired, that $\supp \nu_+ \subset \partial \hat{B}_R$.
\end{proof}
Proposition~\ref{prop:bounded-gds} has an inherent problem in that the
limiting mass configuration $\nu$ seemingly may depend on the choice
of infinitely covering sequence of $\hat{B}_R$, but this is in fact
not the case. To see this, we will establish a characterization of the
odometer function $u$, and hence of the limiting mass configuration
$\nu$ via $\nu = \sigma + \Delta u$, that does not depend on the
choice of infinitely covering sequence; this characterization will
also in fact be our link to the partial balayage operation in the
continuous setting discussed later on in the paper.

To begin with, we need to define a discrete analogue of the potential
function in continuous potential theory. For any given function ${\mu:
\xi \Z^d \to \R}$, assumed to have compact (\ie finite) support, we
define the \emph{(discrete) potential} $U^\mu_\xi$ (or simply $U^\mu$
whenever it is clear which lattice we are referring to) of $\mu$ via
\begin{gather*}
  U^\mu \equiv U^\mu_\xi(x) := \xi^d \sum_{y \in \xi\Z^d}
  g_\xi(x,y) \mu(y).
\end{gather*}
Here $g_\xi(\cdot, \cdot)$ is the discrete Green's function on the
underlying lattice, defined for $x, y \in \xi \Z^d$ by
\begin{align*}
  g_\xi(x, y) := \left\{
    \begin{array}{l l}
      \frac{2}{\pi} \log \xi - \gamma_0 \left(\frac{x}{\xi},
        \frac{y}{\xi} \right) & \text{ if } d = 2, \\
      \frac{1}{\xi^{d-2}} \gamma_1 \left( \frac{x}{\xi},
        \frac{y}{\xi} \right) & \text{ if } d \geq 3,
    \end{array}
  \right.
\end{align*}
where 
\begin{align*}
  \gamma_0(x,y) = \lim_{n \to \infty} (\E_x |\{k \leq n : X_k = x\}| -
  \E_x |\{k \leq n : X_k = y\}|)
\end{align*}
is the (recurrent) potential kernel for simple random walk on $\Z^2$,
and $\gamma_1(x,y)$ is the Green's function for simple random walk on 
$\Z^d$ for $d \geq 3$,
\begin{align*}
  \gamma_1(x, y) = \E_x |\{k : X_k = y\}|.
\end{align*}
Here $\E_x$ denotes expectation with the simple random walk started at
the lattice site $x$; see \cite{lawler-limic-10, lawler-91} for
details on these Green's functions. The above definitions imply in
particular that
\begin{align*}
  -\Delta_1 g_\xi(x,y) = \frac{1}{\xi^d} \delta_x(y) = -\Delta_2 g_\xi(x,y),
\end{align*}
where $\delta_{x,y}$ is the Kronecker delta, and $\Delta_j$ is the
discrete Laplace operator acting on the $j$th variable. As an
immediate and important consequence, it follows that
\begin{align*}
  - \Delta U^\mu_\xi(x) = \xi^d \sum_{y \in \xi \Z^d} (-\Delta_1
  g_\xi(x,y)) \mu(y) = \sum_{y \in \xi \Z^d} \delta_{x,y} \mu(y) =
  \mu(x),
\end{align*}
just as in the continuous setting. In a similar manner, we can via an
easy calculation moreover see that for any function
$v : \xi \Z^d \to \R$ having finite support we have
$-U_\xi^{\Delta v}(x) = v(x)$.

That $u$ indeed is independent of the choice of infinitely covering
sequence now follows from the following proposition:
\begin{proposition}
  \label{prop:gds-as-obst-problem}
  Let $\sigma$ and $R > 0$ be as in
  Proposition~\ref{prop:bounded-gds}, and let $\nu$ and $u$ denote the
  corresponding limit functions relative to toppling of some
  infinitely covering sequence $x_1, x_2, \ldots$ of $\hat{B}_R$. 

  Then $\nu = \sigma + \Delta u$ and $u = U^{\sigma} - v$, where
  \begin{gather}
    \label{eq:gds-def-obst-problem}
    v(x) := \sup\{f(x) : \Delta f \geq 0 \text{ in } \hat{B}_R,\ 
    f \leq U^{\sigma} \text{ in } \xi\Z^d\}.
  \end{gather}
\end{proposition}
\begin{proof}
  We know that $\nu = \sigma + \Delta u$, hence $\Delta u = \nu -
  \sigma$. Let $v' := U^{\sigma} - u$. We immediately obtain
  \begin{align*}
    -\Delta v' = -\Delta U^{\sigma} + \Delta u = \sigma + \nu - \sigma =
    \nu = \nu_+ - \nu_-.
  \end{align*}
  Since $\nu$ is $\nu = -\nu_- \leq 0$ in $\hat{B}_R$, it follows that
  $\Delta v' \geq 0$ in $\hat{B}_R$. Moreover, as $u(x)$ is $\xi^2$
  times the total mass emitted from a site $x \in \xi\Z^d$ it is clear
  that $u \geq 0$ holds everywhere in $\Z^d$, and so
  $v' = U^{\sigma} - u \leq U^{\sigma}$. We conclude that $v'$ is a
  competing function in the definition of $v$ in
  \eqref{eq:gds-def-obst-problem}, which shows that $v \geq v'$ holds
  everywhere in $\xi\Z^d$.
  
  For the converse inequality, let us study the difference $v -
  v'$. First of all, we observe that \eqref{eq:gds-def-obst-problem}
  implies that also the solution $v$ to the obstacle problem will
  satisfy $\Delta v \geq 0$ in $\hat{B}_R$. Indeed, let $f$ be any
  function satisfying both $\Delta f \geq 0$ in $\hat{B}_R$ and
  $f \leq U^{\sigma}$ in $\xi\Z^d$. That $f$ is subharmonic in
  $\hat{B}_R$ means that
  \begin{gather*}
    f(x) \leq \frac{1}{2d} \sum_{y \sim x} f(y)
  \end{gather*}
  holds for every $x \in \hat{B}_R$. Using the inequality $v \geq f$
  on the right hand side implies
  \begin{gather*}
    f(x) \leq \frac{1}{2d} \sum_{y \sim x} v(y),
  \end{gather*}
  and taking supremum over all such functions $f$ on the left hand side
  shows that $\Delta v \geq 0$ must hold everywhere in $\hat{B}_R$.
  
  Now, we have
  \begin{gather*}
    \Delta(v - v') = \Delta v - \Delta v' = \Delta v + \nu.
  \end{gather*}
  For every $x$ belonging to the set
  $D := \{y \in \hat{B}_R: \nu(y) = 0\}$ it is thus clear that
  $\Delta(v-v')(x)=\Delta v(x) \geq 0$, since we just established that
  $v$ is subharmonic in $\hat{B}_R$. On the other hand, for every
  $x \in \hat{B}_R \setminus D$ we must have $\nu(x) < 0$, which
  evidently implies that $x$ must be a site that, during the toppling
  process, never emitted any mass, \ie a site where $u(x) = 0$. Since
  we only do toppling at the sites belonging to $\hat{B}_R$, it is
  moreover clear that $u(x) = 0$ for every $x \notin D$. But for any
  such $x$ we then obtain
  \begin{gather*}
    (v - v')(x) = v(x) - U^{\sigma}(x) + u(x) = v(x) - U^{\sigma}(x)
    \leq 0.
  \end{gather*}
  Hence $v - v'$ is a function that is subharmonic on $D$ and
  satisfies $v - v' \leq 0$ outside $D$, and so the maximum principle
  implies that $v - v' \leq 0$ in fact must hold everywhere on
  $\xi\Z^d$, \ie $v \leq v'$ holds everywhere. We can now finally
  conclude that $v = v'$, hence $u = U^{\sigma} - v$ as stated.
\end{proof}

As seen in the two previous propositions, we obtain for each $R > 0$ a
well-defined generalized mass configuration $\nu = \nu_+ - \nu_-$ as
long as the support of $\sigma$ belongs to $\hat{B}_R$. For sake of
simplicity, we introduce the following notation:
\begin{definition}
  Let $\sigma: \xi\Z^d \to \R$ be a generalized mass configuration and
  let $R > 0$ be such that $\supp \sigma \subset \hat{B}_R$. We call
  the resulting generalized mass configuration $\nu$ in
  Propositions~\ref{prop:bounded-gds} and
  \ref{prop:gds-as-obst-problem} the \emph{generalized divisible
    sandpile configuration of $\sigma$ in $\hat{B}_R$}, and denote
  this configuration ${\gds_R(\sigma) \equiv \gds^\xi_R(\sigma) := \nu}$
  (as a function on $\xi \Z^d$).
\end{definition}

\begin{remark}
  In the previous definition $\gds_R^{\xi}(\sigma)$ is a function
  defined on the same lattice $\xi \Z^d$ as $\sigma$. However, we can
  in a natural way interpret $\gds_R^{\xi}(\sigma)$ as a (signed)
  \emph{measure on $\R^d$} (with some slight abuse of notation):
  \begin{gather*}
    \gds_R^{\xi}(\sigma) = \xi^d \sum_{x \in \xi \Z^d} \gds_R^{\xi}(\sigma)(x)
    \delta_x,
  \end{gather*}
  where $\delta_x$ is the Dirac measure at $x$. That this is
  well-defined follows from the fact that $\gds_R^\xi(\sigma)(x)$ is
  bounded, and zero except for finitely many $x \in \xi \Z^d$.
\end{remark}

\subsection{GDS and energy minimization}

There is a rather natural interpretation of the algorithm for the
generalized divisible sandpile as that minimizing a certain energy. In
the continuous setting, the energy of a measure $\mu$ is often defined
as
\begin{align*}
  I[\mu] = \int_{\R^d} U^\mu \d \mu.
\end{align*}
Following this, we define in the discrete setting the energy
$\En[\eta]$ of a mass configuration $\eta : \xi \Z^d \to \R$ using
\begin{align*}
  \En[\eta] :=& \int_{\R^d}(U^\eta_\xi)^\square(x) \eta^\square(x)
                \d x = \int_{\R^d} U^\eta_\xi(x^{::}) \eta(x^{::}) \d x \\
  =& \xi^d \sum_{y \in \xi \Z^d} U^\eta_\xi(y) \eta(y) = \xi^{2d}
     \sum_{x,y \in \xi \Z^d} g_\xi(x,y) \eta(x) \eta(y).
\end{align*}
For later use, we also define the \emph{mutual energy}
$\En[\sigma, \kappa]$ between two mass configurations $\eta$ and
$\kappa$ defined on the same lattice $\xi \Z^d$ as
\begin{align*}
  \En[\eta, \kappa] := \xi^d \sum_{y \in \xi \Z^d} U^\eta_\xi(y)
  \kappa(y). 
\end{align*}
Note that $\En[\eta, \kappa] = \En[\kappa, \eta]$ and
$\En[\eta] = \En[\eta, \eta]$.

In a rather straightforward way, we can explicitly calculate how the
energy behaves when we perform a toppling in the algorithm for the
generalized divisible sandpile. Let $\sigma$, $x_1, x_2, \ldots$ and
$\sigma_k = T_{x_k} \sigma_{k-1}$ be as in
Proposition~\ref{prop:bounded-gds}, let $\En_k = \En[\sigma_k]$, and
study the difference $\En_k - \En_{k-1}$ in energy between two mass
configurations that only differ in that we have toppled in precisely
one point (the point $x_k$). We write
\begin{align}
  \label{eq:energy-diff-one-toppling}
  \En_k - \En_{k-1} = \xi^{2d} \sum_{x,y \in \xi \Z^d} g_\xi(x,y) d(x,y),
\end{align}
where we let
$d(x,y) := \sigma_k(x) \sigma_k(y) -
\sigma_{k-1}(x)\sigma_{k-1}(y)$. If the mass of $\sigma_{k-1}$ at
the point $x_k$ where we want to topple satisfies
$\sigma_{k-1}(x_k) \leq 0$, then the mass configuration is
unchanged, \ie $\sigma_k = \sigma_{k-1}$ everywhere, hence
$d(x,y) = 0$ for all $x, y \in \xi \Z^d$ and $\En_k =
\En_{k-1}$. Assume therefore that $(\sigma_{k-1})_+(x_k) > 0$, so
that $\sigma_{k-1}$ and $\sigma_k$ are not equal everywhere. In that
case, the double sum in \eqref{eq:energy-diff-one-toppling} can be
split into nine different terms, depending on if $x$ (and similarly
for $y$) is either equal to the toppling point $x_k$, is a neighbour
of $x_k$, or belongs to the set
$S_k := \xi \Z^d \setminus (\{x_k\} \cup \{z : z \sim x_k\})$. We
get
\begin{align*}
  \En_k - \En_{k-1} &= \xi^{2d} \left[ g_\xi(x_k, x_k) d(x_k, x_k) +
                      \sum_{y \sim x_k} g_\xi(x_k, y) d(x_k, y) \right. \\
  +& \sum_{y \in S_k} g_\xi(x_k, y) d(x_k, y) + \sum_{x \sim x_k}
     g_\xi(x, x_k) d(x,x_k) \\
  +& \sum_{x \sim x_k} \sum_{y \sim x_k} g_\xi(x,y) d(x,y) + \sum_{x
     \sim x_k} \sum_{y \in S_k} g_\xi(x,y) d(x,y) \\
  +& \sum_{x \in S_k} g_\xi(x, x_k) d(x, x_k) + \sum_{x \in S_k} \sum_{y
     \sim x_k} g_\xi(x, y) d(x, y) \\
  +& \left. \sum_{x \in S_k}\sum_{y \in S_k} g_\xi(x, y) d(x, y) \right].
\end{align*}
Since both $g_\xi(\cdot, \cdot)$ and $d(\cdot, \cdot)$ are symmetric
functions in their respective arguments, the above can be reduced to
\begin{align*}
  \En_k - \En_{k-1} &= \xi^{2d} \left[ g_\xi(x_k, x_k) d(x_k, x_k) +
                      2 \sum_{x \sim x_k} g_\xi(x, x_k) d(x, x_k) \right. \\
  +& \  2 \sum_{x \in S_k} g_\xi(x, x_k) d(x, x_k) + \sum_{x \sim x_k}
     \sum_{y \sim x_k} g_\xi(x,y) d(x,y) \\
  +& \left. 2 \sum_{x
     \sim x_k} \sum_{y \in S_k} g_\xi(x,y) d(x,y) + \sum_{x \in S_k}\sum_{y \in S_k} g_\xi(x, y) d(x, y) \right],
\end{align*}
which shows that we only have to calculate the combination
$g_\xi(x,y) d(x,y)$ for the six different cases appearing in this
expression:
\begin{itemize}
\item $x = x_k$, $y = y_k$: since we topple at $x_k$ we have
  $\sigma_k(x_k) = 0$, thus
  $d(x_k, x_k) = - \sigma_{k-1}(x_k)^2$,
\item $x \sim x_k, y = x_k$:
  $d(x,x_k) = - \sigma_{k-1}(x_k) \sigma_{k-1}(x)$,
\item $x \in S_k, y = x_k$:
  $d(x,x_k) = - \sigma_{k-1}(x_k) \sigma_{k-1}(x)$,
\item $x \sim x_k, y \sim x_k$:
  $\displaystyle d(x,y) = \frac{\sigma_{k-1}(x_k)}{2d} \cdot \left(\sigma_{k-1}(x) +
    \sigma_{k-1}(y) + \frac{\sigma_{k-1}(x_k)}{2d}\right)$,
\item $x \sim x_k, y \in S_k$: $\displaystyle d(x,y) = \frac{\sigma_{k-1}(x_k) \sigma_{k-1}(y)}{2d}$,
\item $x \in S_k, y \in S_k$: $d(x,y) = 0$.
\end{itemize}
Inserting this into the above and simplifying, once more also using
the symmetric property of $g_\xi(\cdot, \cdot)$, we obtain
\begin{align*}
  \En_k - \En_{k-1} &= \xi^{2d} \sigma_{k-1}(x_k) \left[
                      -g_\xi(x_k, x_k) \sigma_{k-1}(x_k) - 2 \sum_{x \neq x_k}
                      g_\xi(x, x_k) \sigma_{k-1}(x) \right. \\
  +& \left. 2 \frac{1}{2d} \sum_{x \neq x_k} \sum_{y \sim x_k} g_\xi(x,y)
     \sigma_{k-1}(x) + \frac{1}{4d^2} \sum_{x \sim x_k} \sum_{y \sim x_k}
     g_\xi(x,y) \sigma_{k-1}(x_k) \right]
\end{align*}
Two of the four terms vanish, since they can be combined in the
following manner:
\begin{align*}
  - &\sum_{x \neq x_k} g_\xi(x, x_k) \sigma_{k-1}(x) + \frac{1}{2d}
      \sum_{x \neq x_k} \sum_{y \sim x_k} g_\xi(x, x_k) \sigma_{k-1}(x)
  \\
    &= \sum_{x \neq x_k} \sigma_{k-1}(x) \left[ -g_\xi(x,x_k) +
      \frac{1}{2d} \sum_{y \sim x_k} g_\xi(x,y) \right] \\
    &= \sum_{x \neq x_k} \sigma_{k-1}(x) \underbrace{ \left[\frac{1}{2d} \sum_{y
      \sim x_k} (g_\xi(x,y) - g_\xi(x,x_k) \right]}_{ = \xi^2 \Delta_2
      g_\xi(x,x_k) = -\xi^{2-d} \delta_{x_k}(x)} \\
    &= - \xi^{2-d} \sum_{x \neq x_k} \sigma_{k-1}(x) \delta_{x_k}(x) = 0.
\end{align*}
As for the two remaining terms in $\En_k - \En_{k-1}$, we see in a
similar way that
\begin{align*}
  -g_\xi(x_k, &x_k) + \frac{1}{4d^2} \sum_{x \sim x_k} \sum_{y \sim
                x_k} g_\xi(x, y) \\
              &= \frac{1}{2d} \sum_{x \sim x_k}
                \left[-g_\xi(x_k, x_k) + g_\xi(x,x_k) - g_\xi(x,x_k) +
                \frac{1}{2d} \sum_{y \sim x_k} g_\xi(x,y) \right] \\
              &= \xi^2 \Delta_1 g_\xi(x_k,x_k) + \frac{1}{2d} \sum_{x \sim x_k}
                \underbrace{\frac{1}{2d} \sum_{y \sim x_k} \left(
                g_\xi(x,y) - g_\xi(x,x_k) \right)}_{=\xi^2 \Delta_2
                g_\xi(x,x_k)} \\
              &= -\xi^{2-d} \delta_{x_k}(x_k) - \xi^{2-d} \frac{1}{2d} \sum_{x
                \sim x_k} \delta_{x_k}(x) = - \xi^{2-d},
\end{align*}
from which it immediately finally follows that
\begin{align*}
  \En_k - \En_{k-1} = -\xi^{2d} \sigma_{k-1}(x_k)^2 \xi^{2-d} = -
  \xi^{d+2} (\sigma_{k-1})_+(x_k)^2,
\end{align*}
\ie whenever it happens that
$(\sigma_{k-1})_+(x_k)$ is positive at toppling step
$k$, then the energy strictly decreases. The total energy after
$k$ steps is
\begin{align}
  \En_k = \En_0 - \xi^{d+2} \sum_{j=1}^k (\sigma_{j-1})_+(x_j)^2.
\end{align}
Comparing this with \eqref{eq:sum-of-uk-is-sum-of-toppled-mass}
immediately shows that $\En_k$ has a finite limit as $k \to \infty$.

Now, consider the problem of finding a mass configuration
$\tilde{\nu}$ with the properties $\tilde{\nu} \leq
0$ in $\hat{B}_R$ and with the same total mass as
$\sigma$, that minimizes energy of the difference between
$\sigma$ and $\tilde{\nu}$, \ie that solves the problem
\begin{align*}
  \min \En[\sigma - \tilde{\nu}] : \tilde{\nu} \leq 0 \text{ in }
  \hat{B}_R \text{ and } \sum_{y \in \xi \Z^d} \tilde{\nu}(y) =
  \sum_{y \in \xi \Z^d} \sigma(y).
\end{align*}
We claim that $\nu :=
\gds^\xi_R(\sigma)$ is the (unique) solution to this problem. By
usual Hilbert space theory arguments, it suffices to show that
\begin{align*}
  \En[\sigma - \nu, \nu - \tilde{\nu}] \geq 0
\end{align*}
holds for all $\tilde{\nu}$ with $\tilde{\nu} \leq 0$ in
$\hat{B}_R$ and $\sum_{y \in \xi \Z^d} \tilde{\nu}(y) = \sum_{y \in
  \xi \Z^d} \sigma(y)$. To begin with, we have
\begin{align*}
  \En[\sigma - \nu, \nu - \tilde{\nu}] = \xi^d \sum_{y \in \xi \Z^d}
  (U^\sigma(y) - U^\nu(y)) (\nu(y) - \tilde{\nu}(y)).
\end{align*}
By the definition of $\nu = \sigma + \Delta u$, where
$u$ is the limiting odometer function, it follows that $U^\nu =
U^\sigma -
u$, \ie the first factor in the sum above is precisely $U^\sigma -
U^\nu =
u$. It follows that we may reduce the set over which we sum to the
set of points where $u$ is non-zero, \ie $\{y \in \xi \Z^d : u(y) >
0\}$ (which is a subset of $\hat{B}_R$). However, if $u(y) >
0$ then some mass must have been emitted from
$y$ in the construction of $\gds_R^\xi(\sigma)$, thus $\nu(y) =
0$ must hold. We then obtain
\begin{align*}
  \En[\sigma - \nu, \nu - \tilde{\nu}] = \xi^d \sum_{y : u(y) > 0}
  u(y) (- \tilde{\nu}(y)) \geq 0,
\end{align*}
since both $u$ and
$-\tilde{\nu}$ are non-negative. We can in fact calculate an
explicit expression for the minimizing energy by studying
$\En[\sigma - \sigma_k]$ and letting ${k \to
  \infty}$. For the difference $\En[\sigma - \sigma_k] - \En[\sigma
-
\sigma_{k-1}]$ between two successive steps in the algorithm we get
\begin{align*}
  \En[\sigma - \sigma_k] - \En[\sigma - \sigma_{k-1}] =
  \En[\sigma_k] - \En[\sigma_{k-1}] - 2\En[\sigma, \sigma_k -
  \sigma_{k-1}].
\end{align*}
We already know that $\En[\sigma_k] - \En[\sigma_{k-1}] = -\xi^{d+2}
(\sigma_{k-1})_+(x_k)^2$. For the last term, we get
\begin{align*}
  \En[\sigma,\sigma_k - \sigma_{k-1}] = \En[\sigma_k -
  \sigma_{k-1},\sigma] = \xi^d \sum_{x \in \xi \Z^d} U^{\sigma_k -
  \sigma_{k-1}}(x) \sigma(x),
\end{align*}
and, utilizing that $(\sigma_k - \sigma_{k-1})(x) =
(\sigma_{k-1})_+(x_k) \xi^2 \Delta \delta_{x_k}(x)$, hence
\begin{align*}
  U^{\sigma_k - \sigma_{k-1}}(x) &= (\sigma_{k-1})_+(x_k) \xi^2 U^{\Delta
                                   \delta_{x_k}}(x) \\
                                 &= (\sigma_{k-1})_+(x_k) \xi^2 \Delta
                                   U^{\delta_{x_k}}(x) = - (\sigma_{k-1})_+(x_k) \xi^2 \delta_{x_k}(x),
\end{align*}
it follows that
\begin{align*}
  \En[\sigma,\sigma_k - \sigma_{k-1}] = - \xi^{d+2}
  (\sigma_{k-1})_+(x_k) \sigma(x_k).
\end{align*}
We can summarize the above to draw the conclusion
\begin{align*}
  \En[\sigma - \sigma_k] - \En[\sigma - \sigma_{k-1}] &= - \xi^{d+2}
                                                        (\sigma_{k-1})_+(x_k)^2 + 2 \xi^{d+2} (\sigma_{k-1})_+(x_k) \sigma(x_k) \\
                                                      &= \xi^{d+2} (\sigma_{k-1})_+(x_k) (2 \sigma(x_k) - (\sigma_{k-1})_+(x_k)),
\end{align*}
hence
\begin{align*}
  \En[\sigma - \sigma_k] &= \sum_{j=1}^k(\En[\sigma - \sigma_j] -
                           \En[\sigma - \sigma_{j-1}]) \\
                         &= \xi^{d+2}
                           \sum_{j=1}^k(\sigma_{j-1})_+(x_j) (2 
                           \sigma(x_j) - (\sigma_{j-1})_+(x_j)).
\end{align*}
From this it follows that the minimizing energy is precisely
\begin{align}
  \En[\sigma - \nu] = \xi^{d+2} \sum_{j=1}^\infty(\sigma_{j-1})_+(x_j)
  (2 \sigma(x_j) - (\sigma_{j-1})_+(x_j)).
\end{align}
Note that this is convergent, as the factor $2\sigma(x_j) -
(\sigma_{j-1})_+(x_j)$ is bounded and the sum $\xi^2
\sum_{j=1}^\infty
(\sigma_{j-1})_+(x_j)$ is by
\eqref{eq:sum-of-uk-is-sum-of-toppled-mass} precisely equal to
\begin{align*}
  \sum_{x \in \xi \Z^d} u(x) = \sum_{x \in \hat{B}_R} u(x) < \infty.
\end{align*}

\subsection{A natural scaling limit of the bounded GDS}
\label{sec:sc-lim-bdd-gds}

As mentioned in the introduction, we are interested in studying all of
the above in the natural scaling limit, \ie as the lattice spacing
tends to zero. For this reason, we simply fix a sequence of positive
real numbers $\{ \xi_n \}_{n=1}^\infty$, which is assumed to be
monotonically decreasing and with limit zero as $n$ tends to
infinity. Our initial generalized mass configuration $\sigma$ is now
assumed to be a bounded function defined on $\R^d$ instead of some
lattice, and for each lattice constant $\xi_n$ we now discretize
$\sigma$ in precisely the same way as in
Section~\ref{sec:prel-main-result}, \ie we define for each
$n = 1, 2, \ldots$ the function $\sigma_n : \xi_n \Z^d \to \R$ via
\begin{align}
  \label{eq:sigma-sc-def}
  \sigma_n(x) := \frac{1}{\xi_n^d} \int_{x^\square} \sigma(y) \d y.
\end{align}
For each $n$ we thus obtain a generalized mass configuration on a
lattice, can perform the generalized divisible sandpile algorithm on
each such configuration, and hence will obtain a sequence of
generalized mass configurations
$\{\gds_R^{\xi_n}(\sigma_n)\}_{n=1}^\infty$ (for some $R$ chosen in a
suitable manner). Note that the discretization above comes at a
(slight) price: in general we do not necessarily have
$(\sigma_n)_+ = (\sigma_+)_n$ or $(\sigma_n)_- = (\sigma_-)_n$, only
in the limit $n \to \infty$.

We claim the following:
\begin{theorem}
  \label{thm:bdd-gds-conv-to-pb}
  Let $\sigma : \R^d \to \R$ be a bounded and almost everywhere
  continuous function with compact support for which
  $\int_{\R^d} \sigma(x) \d x < 0$, let $\{\xi_n\}_{n=1}^\infty$ be a
  sequence of positive decreasing lattice constants such that
  $\xi_n \searrow 0$ as $n \to \infty$, and for each $n=1,2,\ldots$
  let $\sigma_n : \xi_n \Z^d \to \R$ be the discretization of $\sigma$
  relative to $\xi_n \Z^d$ as in \eqref{eq:sigma-sc-def}. Assume
  $R > 0$ is such that $\supp \sigma \subset B(0,R)$ and
  $\supp \sigma_n \subset B(0,R)$ for all $n$. Then, in the sense of distributions,
  \begin{align}
    \label{eq:bdd-gds-conv-to-pb}
    \gds_R^{\xi_n}(\sigma_n) \to \bal_R(\sigma, 0) \text{ as } n \to \infty.
  \end{align}
\end{theorem}

\noindent To prove this theorem we need a few lemmas.

\begin{lemma}
  \label{lemma:unif-conv-potentials-compact-subsets}
  Let $U^{\sigma_n}_n \equiv U^{\sigma_n}_{\xi_n}$ be the discrete
  potential of $\sigma_n$, defined on $\xi_n \Z^d$, let
  $(U^{\sigma_n}_n)^\square$ be its extension to $\R^d$ as a step
  function and let $U^\sigma$ be the potential of the measure
  $\sigma(x) \d x$. Then $(U^{\sigma_n}_n)^\square \to U^\sigma$
  uniformly on compact subsets of $\R^d$ as $n \to \infty$.
\end{lemma}

For the proof of
Lemma~\ref{lemma:unif-conv-potentials-compact-subsets} we refer to the
proofs of Lemma~2.16\,(i) and Lemma~2.22 in~\cite{levine-peres-10}
which, although there stated with slightly different assumptions than
the ones in this paper, go through in our setting as well, with more
or less only notational changes.

\begin{lemma}
  \label{lemma:conv-of-odometers-in-bdd-scaling-limit}
  Let $u_n$ be the limiting odometer function for the generalized
  divisible sandpile on $\xi_n \Z^d$ from mass configuration
  $\sigma_n$, and let $u = U^\sigma - V^{\sigma}$ be the modified
  Schwarz potential of $\bal_R(\sigma, 0)$ as in
  Remark~\ref{remark:mod-schwarz-potential}, with $\sigma$ and
  $\sigma_n$ as in Theorem~\ref{thm:bdd-gds-conv-to-pb}. Then for
  every $x \in \R^d$, $(u_n)^\square(x) \to u(x)$ pointwise as
  $n \to \infty$.
\end{lemma}

\begin{proof}
  Let us first restrict the problem slightly. We know that the
  function $u$ is zero on the complement of $B(0,R)$, and for each $n$
  we also know that the odometer function $u_n$ is zero outside of the
  set $\hat{B}_R^{(n)} := B(0,R) \cap (\xi_n \Z^d)$. For any
  $x \notin B(0,R)$ it therefore follows that for all $n$ large enough
  we have $(u_n)^\square(x) = u_n(x^{::}) = 0 = u(x)$. The set we have
  to study in detail is thus $B(0,R)$. The slightly more challenging
  part of the proof is thus the convergence for $x$ in the set
  $B(0,R)$.

  We mainly repeat the arguments made in the proof of Lemma~3.8
  in~\cite{levine-peres-10}, with a few modifications due to the fact
  that we here work in a slightly different setting, being bounded to
  the set $B(0,R)$. As a first step, we use that
  $u = U^\sigma - V^{\sigma, R}$, $u_n = U_n^{\sigma_n} - v_n$ along
  with the convergence $(U_n^{\sigma_n})^\square \to U^\sigma$ from
  Lemma~\ref{lemma:unif-conv-potentials-compact-subsets} to conclude
  that it suffices to show that $(v_n)^\square(x) \to V^{\sigma}(x)$
  for all $x \in B(0,R)$, where
  \begin{align}
    \label{eq:main-lemma--cont-obst-problem-solution}
    V^{\sigma, R}(x) = \sup\{ f(x) : f \in \mathcal{CS}(B(0,R)),\ f \leq U^\sigma \text{ in } \R^d\},
  \end{align}
  the set $\mathcal{CS}(B(0,R))$ is the set of functions on $\R^d$
  that are continuous and subharmonic on $B(0,R)$, and
  \begin{align}
    \label{eq:main-lemma--disc-obst-problem-solution}
    v_n(x) := \sup\{f(x) : \Delta f \geq 0 \text{ in }
    \hat{B}_R^{(n)},\ f \leq U_n^{\sigma_n} \text{ in } \xi_n \Z^d\}.
  \end{align}
  The method we will employ will in essence be to construct help
  functions that are comparable to $V^{\sigma, R}$ and $v_n$,
  respectively, but have discrete or continuous analogues that are
  competing functions in the obstacle problems
  \eqref{eq:main-lemma--cont-obst-problem-solution} and
  \eqref{eq:main-lemma--disc-obst-problem-solution}, thereby allowing
  us to conclude both $V^{\sigma, R}(x) \leq (v_n)^\square(x)$ and
  ${(v_n)^\square(x) \leq V^{\sigma, R}(x)}$ for $n$ large enough.

  Let $\epsilon > 0$ be arbitrary but fixed. We want to show that
  \begin{align*}
    (v_n)^\square(x) \geq V^{\sigma}(x)
  \end{align*}
  holds for all $n$ large enough and all $x \in B(0,R)$. For any
  $h > 0$ let ${\tilde{V}^\sigma := \mathbf{J}_h V^\sigma}$ be the
  mollification of $V^\sigma$ (for instance as
  in~\cite[Section~3.5]{helms-14}):
  \begin{align*}
    \tilde{V}^\sigma(x) := \mathbf{J}_h V^\sigma(x) = \frac{1}{h^d}
    \int_{\R^d} V^\sigma(y) m \left( \frac{x-y}{h} \right) \d y, 
  \end{align*}
  where $m(y) = C \exp(-1/(1-|y|^2))$ if $|y| < 1$ and zero otherwise,
  with $C$ such that $\int m(y) \d y = 1$. By taking $h$ small enough
  we obtain $|V^\sigma - \tilde{V}^\sigma| < \epsilon$ on $B(0,R-h)$,
  in particular $\tilde{V}^\sigma(x) + \epsilon > V^\sigma(x)$ for all
  $x \in B(0, R-h)$. We will construct our helper function from the
  discretization $(\tilde{V}^{\sigma})^{::}$ of $\tilde{V}^{\sigma}$,
  and need to relate the discrete Laplacian of
  $(\tilde{V}^{\sigma})^{::}$ to the continuous Laplacian of
  $\tilde{V}^\sigma$ (which is well-defined since $\tilde{V}^\sigma$
  is infinitely differentiable). In general, a straightforward
  calculation (for instance in~\cite[Lemma~2.20]{levine-peres-10})
  shows that if $f \in C^\infty(D)$ on some open set $D \subset \R^d$,
  $A$ is a bound for the third derivative of $f$ in $D$, and
  $x \in D \cap \xi \Z^d$ with $B(x, \xi) \subset D$, then
  \begin{align*}
    |\Delta f(x) - 2d \Delta f^{::}(x)| \leq \frac{A d}{3} \xi.
  \end{align*}
  For any fixed value of $n$, note that we can always choose $h > 0$
  small enough so that the set $\hat{B}_R^{(n)}$ is contained in
  $B(0, R - h)$. Let $A$ be a bound for the third partial derivatives
  of $\tilde{V}^\sigma$ in $B(0,R-h)$, and let
  $\phi_n : \xi_n \Z^d \to \R$ be defined by
  \begin{align*}
    \phi_n(x) := (\tilde{V}^\sigma)^{::}(x) + \frac{A \xi_n
    |x|^2}{6}.
  \end{align*}
  It follows that for all $x \in \hat{B}_R^{(n)}$,
  $\Delta \phi_n(x) \geq \Delta \tilde{V}^\sigma(x) / 2d$. However,
  $V^\sigma$ is subharmonic in $B(0,R)$, hence $\tilde{V}^\sigma$ is
  subharmonic in $B(0, R-h)$, and thus $\phi_n$ is (discrete)
  subharmonic in $\hat{B}_R^{(n)}$. If $n$ is large enough then the
  term $A \xi_n |x|^2 / 6$ is strictly less than $\epsilon$ in
  $\hat{B}_R^{(n)}$, from which it follows that
  \begin{align*}
    \phi_n(x) < (V^\sigma)^{::}(x) + 2 \epsilon
  \end{align*}
  for all $x \in \hat{B}_R^{(n)}$. Since $V^\sigma$ is bounded from
  above by $U^\sigma$, and we again use the property
  $|(U_n^{\sigma_n})^\square - U^\sigma| < \epsilon$ for all $n$ large
  enough by Lemma~\ref{lemma:unif-conv-potentials-compact-subsets}, we
  obtain
  \begin{align*}
    \phi_n(x) - 3 \epsilon \leq U_n^{\sigma_n}(x)
  \end{align*}
  for all $x \in \hat{B}_R^{(n)}$. Now define $\Phi_n : \xi_n \Z^d \to
  \R$ via
  \begin{align*}
    \Phi_n(x) := \left\{
    \begin{array}{l l}
      \phi_n(x) - 3 \epsilon & \text{ if } x \in \hat{B}_R^{(n)}, \\[.2cm]
      U_n^{\sigma_n}(x) & \text{ otherwise}.
    \end{array}
        \right.
  \end{align*}
  It follows that $\Phi_n$ is a lattice function that is subharmonic
  on $\hat{B}_R^{(n)}$ and satisfies $\Phi_n \leq U_n^{\sigma_n}$
  everywhere on $\xi_n \Z^d$. The function $\Phi_n$ is thus a
  competing element in the obstacle problem
  \eqref{eq:main-lemma--disc-obst-problem-solution}, hence
  $\Phi_n \leq v_n$ holds everywhere on $\xi_n \Z^d$. For any $x \in
  B(0,R)$ we can now conclude that
  \begin{align}
    (v_n)^\square(x) = v_n(x^{::}) &\geq \Phi_n(x^{::}) \geq
                                     (\tilde{V}^\sigma)^{::}(x^{::}) - 3 \epsilon \nonumber \\
                                   &> V^\sigma(x^{::}) - 4 \epsilon > V^\sigma(x) - 5 \epsilon, \label{eq:main-lemma--lower-bound}
  \end{align}
  where we in the last equality used that for all $n$ large enough we
  have ${|(V^\sigma)^\square - V^\sigma| < \epsilon}$.

  For the converse result, we again repeat the techniques in the proof
  of Lemma~3.8 in~\cite{levine-peres-10}: let $R' > R$ and introduce
  the function $\psi_n : \xi_n \Z^d \to \R$ defined by
  \begin{align*}
    \psi_n(x) := - \Delta (v_n \chi_{\hat{B}_{R'}})(x).
  \end{align*}
  On the one hand we have
  $U_n^{\psi_n}(x) = v_n(x) \chi_{\hat{B}_{R'}}(x)$ ($ = v_n(x)$ for
  $x \in \hat{B}_{R'}$). It can be shown (see the proof of Lemma~3.7
  in~\cite{levine-peres-10}, the same methods apply here) that we have
  a similar property for $x \in B(0,R)$ if we try to take the
  continuous potential of the function $(\psi_n)^\square$ considered
  as a measure on $\R^d$ (in the sense that
  $d(\psi_n)^\square(x) = (\psi_n)^\square(x) \d m(x)$, where $m$ is the
  Lebesgue measure on $\R^d$): for any $\epsilon > 0$ we have 
  $|(v_n)^\square(x) - U^{(\psi_n)^\square}(x)| < \epsilon$ for all
  $x \in B(0,R)$ if $n$ is large enough. Assuming $n$ is also large
  enough for $|(U_n^{\sigma_n})^\square(x) - U^\sigma(x)| < \epsilon$
  to hold for all $x \in B(0,R)$, we obtain
  \begin{align*}
    U^{(\psi_n)^\square}(x) < (v_n)^\square(x) + \epsilon \leq
    (U_n^{\sigma_n})^\square(x) + \epsilon < U^\sigma(x) + 2 \epsilon.
  \end{align*}
  Let $\Psi_n : \R^d \to \R$ be defined by
  \begin{align*}
    \Psi_n(x) := \left\{
    \begin{array}{l l}
      U^{(\psi_n)^\square}(x) - 2 \epsilon & \text{ if } x \in B(0,R),
      \\[.2cm]
      U^\sigma(x) & \text{ otherwise}.
    \end{array}
        \right.
  \end{align*}
  The function $\Psi_n$ is then subharmonic and continuous in
  $B(0,R)$, since $(\psi_n)^\square$ is non-positive there and
  bounded. By the above, we clearly also have ${\Psi_n \leq U^\sigma}$
  everywhere in $\R^d$. It immediately follows that $\Psi_n$ is a
  competing function in the obstacle problem
  \eqref{eq:main-lemma--cont-obst-problem-solution}, hence satisfies
  $\Psi_n(x) \leq V^\sigma(x)$ everywhere on $\R^d$. In particular,
  for $x \in B(0,R)$ this implies
  \begin{align}
    \label{eq:main-lemma--upper-bound}
    (v_n)^\square(x) < U^{(\psi_n)^\square}(x) + \epsilon = \Psi_n(x)
    + 3 \epsilon \leq V^\sigma(x) + 3 \epsilon.
  \end{align}
  Finally, combining \eqref{eq:main-lemma--lower-bound} and
  \eqref{eq:main-lemma--upper-bound}, we can conclude that if
  $\epsilon' > 0$ is arbitrary and $x \in B(0,R)$ then there exists
  $N$ such that we for all $n > N$ have both
  $(v_n)^\square(x) > V^\sigma(x) - \epsilon'$ and
  $(v_n)^\square(x) < V^\sigma(x) + \epsilon'$, \ie precisely
  \begin{align*}
    |(v_n)^\square(x) - V^\sigma(x)| < \epsilon',
  \end{align*}
  which completes the proof.
\end{proof}

\begin{lemma}
  \label{lemma:conv-of-discr-Laplacian-in-distr-sense}
  Let $\{g_n\}_{n=1}^\infty$ be a sequence of functions with
  $g_n : \xi_n \Z^d \to \R$, for some fixed lattice constants
  $\xi_n$ satisfying $\xi_n \searrow 0$ as $n \to \infty$, and
  assume that $(g_n)^\square(x)$ converges to $g(x)$ for every
  $x \in \R^d$ for some function $g \in
  L^1_{\mathrm{loc}}(\R^d)$. Then
  $(\Delta g_n)^\square \to \Delta g / 2d$ in the sense of
  distributions as $n \to \infty$.
\end{lemma}

\begin{proof}
  Let $\varphi \in C^\infty_0(\R^d)$ be an arbitrary test function. We
  obtain
  \begin{align*}
    &\left< (\Delta g_n)^\square, \varphi \right> = \int_{\R^d} (\Delta g_n)(x^{::}) \varphi(x) \d x = \sum_{y \in \xi_n \Z^d} \int_{y^\square} (\Delta g_n)(y) \varphi(x) \d x \\ 
    &\quad = \frac{1}{2d \xi_n^2} \sum_{y \in \xi_n \Z^d} \sum_{z \sim y} (g_n(z) - g_n(y)) \int_{y^\square} \varphi(x) \d x \\
    &\quad = \frac{1}{2d \xi_n^2} \sum_{y \in \xi_n \Z^d} \sum_{k = 1}^d (g_n (y + \xi_n e_k) - 2 g_n(y) + g_n(y - \xi_n e_k)) \int_{y^\square} \varphi(x) \d x
  \end{align*}
  Utilizing that we sum over the entire lattice $\xi_n \Z^d$ we can
  rewrite this last expression as
  \begin{align*}
    \frac{1}{2d \xi_n^2} &\sum_{y \in \xi_n \Z^d} g_n(y) \int_{y^\square} \sum_{k = 1}^d (\varphi(x + \xi_n e_k) - 2 \varphi(x) + \varphi(x - \xi_n e_k)) \d x \\
                         &= \frac{1}{2d} \int_{\R^d} g_n(x^{::}) \sum_{k = 1}^d \frac{\varphi(x + \xi_n e_k) - 2 \varphi(x) + \varphi(x - \xi_n e_k)}{\xi_n^2} \d x \\
                         &\to \frac{1}{2d} \int_{\R^d} g(x) \sum_{k=1}^d \frac{\partial^2 \varphi}{\partial x_k^2} \d x = \left< \frac{1}{2d} \Delta g, \varphi \right> \text{ as } n \to \infty,
  \end{align*}
  where the last convergence follows from the dominated convergence
  theorem.
\end{proof}

\begin{proof}[Proof of Theorem~\ref{thm:bdd-gds-conv-to-pb}]
  We know that $\gds^{\xi_n}_R(\sigma_n) = \sigma_n + \Delta u_n$, so
  that
  \begin{align*}
    \gds_R^{\xi_n}(\sigma_n) = (\gds^{\xi_n}_R(\sigma_n))^\square =
    (\sigma_n)^\square + (\Delta u_n)^\square.
  \end{align*}
  Since $\{\sigma_n\}_{n=1}^\infty$ is assumed to be a discretization of $\sigma$,
  we have $(\sigma_n)^\square \to \sigma$ as $n \to \infty$.
  Lemma~\ref{lemma:conv-of-discr-Laplacian-in-distr-sense} combined
  with Lemma~\ref{lemma:conv-of-odometers-in-bdd-scaling-limit} yields
  $(\Delta u_n)^\square \to \Delta u / 2d$. It follows that
  \begin{align*}
    \gds_R^{\xi_n}(\sigma_n) = (\sigma_n)^\square + (\Delta u_n)^\square \to
    \sigma + \frac{\Delta u}{2d} = \bal_R(\sigma, 0),
  \end{align*}
  as desired.
\end{proof}

\subsection{Boundary properties for large confining radii}
\label{sec:bound-prop-grow}
Given the recent development in~\cite{roos-15} of partial balayage in
an unrestricted setting (at least in the plane) described in
Section~\ref{sec:unrestr-part-balayage}, one might expect there to be
a result similar to Theorem~\ref{thm:bdd-gds-conv-to-pb} if we attempt
to study the limit $R \to \infty$. For instance, it is rather easy to
show that there is a sort of invariance in the choice of the confining
radius in the generalized divisible sandpile, in the sense that
successive applications of $\gds_\rho(\cdot)$ operators for, say,
first $\rho = R_1$ and then $\rho = R_2$ for some $0 < R_1 < R_2$,
yields the same result as if we would have used $\rho = R_2$ from the
start. In view of Theorem~\ref{thm:bdd-gds-conv-to-pb} this is
natural, since it is known that a similar iterative property holds for
$\bal_\rho(\cdot, 0)$~\cite[Theorem~2.2\,(iii)]{gustafsson-sakai-94}.
\begin{proposition}
  \label{prop:iterated-balayage}
  If $\sigma: \xi \Z^d \to \R$ is a generalized mass configuration and
  $R_1 > 0$ is such that $\supp \sigma \subset \hat{B}_{R_1}$, and
  $R_2 > R_1 + \xi$ is arbitrary, then
  \begin{gather*}
    \gds_{R_2}(\gds_{R_1}(\sigma)) = \gds_{R_2}(\sigma).
  \end{gather*}
\end{proposition}
\begin{proof}
  For sake of simplicity, we define the three mass configurations
  $\nu_1, \nu_2$ and $\tilde{\nu}$ via
  \begin{gather*}
    \nu_1 := \gds_{R_1}(\sigma) = -\Delta v_1, \\
    \nu_2 := \gds_{R_2}(\sigma) = -\Delta v_2,, \\
    \tilde{\nu} := \gds_{R_2}(\nu_1) = -\Delta \tilde{v},
  \end{gather*}
  where $v_1$, $v_2$ and $\tilde{v}$ are, by
  \eqref{eq:gds-def-obst-problem}, the solutions to the obstacle
  problems
  \begin{gather*}
    v_1(x) = \sup\{f(x) : \Delta f \geq 0 \text{ in } \hat{B}_{R_1},
    f \leq U^{\sigma} \text{ in } \xi\Z^d\}, \\
    v_2(x) = \sup\{f(x) : \Delta f \geq 0 \text{ in }
    \hat{B}_{R_2}, f \leq U^{\sigma} \text{ in } \xi\Z^d\}, \\
    \tilde{v}(x) = \sup\{f(x) : \Delta f \geq 0 \text{ in }
    \hat{B}_{R_2}, f \leq U^{\nu_1} \text{ in } \xi \Z^d\}.
  \end{gather*}
  Note that $\tilde{\nu} = \gds_{R_2}(\nu_1)$ is well-defined since
  $\nu_1$ is a mass configuration of negative total mass satisfying
  $\supp \nu_1 \subset (\bdry \hat{B}_{R_1} \cup \hat{B}_{R_1})
  \subset \hat{B}_{R_2}$ by our assumption on $R_2$. We claim that the
  solutions $\tilde{v}$ and $v_2$ above are in fact equal, from which
  $\tilde{\nu} = \nu_2$ clearly will follow, proving the proposition.

  First of all, by the definitions of $\tilde{v}$ and $v_1$ we see that
  \begin{align*}
    \tilde{v} \leq U^{\nu_1} = U^{-\Delta v_1} = v_1 \leq U^{\sigma}
  \end{align*}
  holds throughout $\xi\Z^d$. Moreover, as was seen in the proof of
  Proposition~\ref{prop:gds-as-obst-problem}, it is clear that
  $\Delta \tilde{v} \geq 0$ holds in $\hat{B}_{R_2}$. Combining this,
  we see that $\tilde{v}$ is a competing function in the definition of
  $v_2$, from which it follows that $\tilde{v} \leq v_2$ holds
  everywhere.

  For the contrary, we know similarly that $v_2 \leq U^{\sigma}$ holds
  everywhere and that $\Delta v_2 \geq 0$ holds in
  $\hat{B}_{R_2}$. But $\hat{B}_{R_2} \supset \hat{B}_{R_1}$ by
  assumption, hence $v_2$ is a competing function in the definition of
  $v_1$, yielding $v_2 \leq v_1$ everywhere. Since
  $v_1 = U^{-\Delta v_1} = U^{\nu_1}$ it thus follows that
  $v_2 \leq U^{\nu_1}$, and so we can conclude that $v_2$ is a
  competing function in the definition of $\tilde{v}$, finally
  yielding $v_2 \leq \tilde{v}$ everywhere in $\xi \Z^d$, and we are
  done.
\end{proof}
With Proposition~\ref{prop:iterated-balayage} in mind, we can also
observe that the total mass of $(\gds_{R_2}(\sigma))_+$, \ie the total
mass of $\gds_{R_2}(\sigma)$ that resides on $\bdry \hat{B}_{R_2}$, in
fact always must be \emph{strictly} less than the total mass residing
on $\bdry \hat{B}_{R_1}$ for $\gds_{R_1}(\sigma)$. (When calculating
$\gds_{R_2}(\sigma) = \gds_{R_2}(\gds_{R_1}(\sigma))$ me must topple
all the points on $\bdry \hat{B}_{R_1}$, with the consequence that at
least a fraction of that mass has to move inwards into the region
where $\gds_{R_1}(\sigma)$ is negative, thereby annihilating and
resulting in that the positive part of $\gds_{R_2}(\sigma)$ must have
strictly less total mass than the positive part of
$\gds_{R_1}(\sigma)$). One might therefore guess that this boundary
mass would vanish if we keep increasing the confining radius $R$, \ie
let $R \to \infty$. However, there does not seem to be any reason for
such a result to hold in general, at least not for dimensions
$d \geq 3$. In the upcoming paper~\cite{gustafsson-roos-16} the
example $\bal_R(\sigma, 0)$ is treated in detail, where $\sigma$ is
the measure
\begin{align*}
  \sigma = t \eta - \chi_{B(0, \rho)},
\end{align*}
and $t > 0$ is a parameter, $\eta$ is the hypersurface measure on the
unit sphere $\bdry B(0, 1)$, and $\chi_{B(0, \rho)}$ is interpreted as
the characteristic function of the set $B(0, \rho)$ times the Lebesgue
measure in $\R^d$ (or, equivalently, the restriction of the Lebesgue
measure to $B(0, \rho)$, extended with zero outside of $B(0,
\rho)$). The two radii $\rho$ and $R$ appearing in the problem are
assumed to satisfy ${0 < \rho < 1 < R}$. If $t$ and $\rho$ are chosen
suitably then $\bal_R(\sigma, 0)$ exists, and by the radial symmetry
of the problem it is possible to explicitly calculate the part of
$\bal_R(\sigma, 0)$ that is supported on the boundary $\bdry B(0,R)$,
\ie the positive part of $\bal_R(\sigma, 0)$. In particular, if we
write ${\nu := \bal_R(\sigma, 0) = \nu_+^{(R)} - \nu_-^{(R)}}$, so that
$\supp \nu_+^{(R)} \subset \bdry B(0,R)$, then the quantity
$M_R := \nu_+^{(R)}(\R^d)$, \ie the total mass residing at the
boundary $\bdry B(0,R)$, has a limit
\begin{align*}
  \lim_{R \to \infty} M_R = M_\infty := \frac{d-2}{d}
  \cdot |S^{d-1}|\, t, 
\end{align*}
where $|S^{d-1}|$ is the surface area of the unit sphere in $\R^d$. By
this example it therefore seems rather likely that any attempt of
finding an unbounded version of Theorem~\ref{thm:bdd-gds-conv-to-pb}
would be rather futile. Also, it is noteworthy that if the dimension
$d$ is very large then in the above example nearly \emph{all} of the
total mass of $\sigma_+$ would be relocated out to the boundary
$\bdry B(0,R)$ by $\bal_R(\cdot, 0)$ for large $R$, suggesting that
the boundary has a rather important impact on the problem for
$d \geq 3$.

In dimension $d = 2$ however, the situation seems slightly
different. For the above example it turns out that the total mass
$M_R$ on the boundary tends to zero as $R \to \infty$. In fact, in
\cite{gustafsson-roos-16} it is shown that the boundary mass vanishes
in general in dimension $d = 2$.  As already mentioned in
Section~\ref{sec:unrestr-part-balayage}, in~\cite{roos-15} it is shown
that one can define a partial balayage operation $\bal(\sigma, 0)$
(see Definition~\ref{def:partial-balayage-unbdd}), in some sense
corresponding to letting $R$ be $R = \infty$ in
Definition~\ref{def:partial-balayage-bdd}. As seen in
Theorem~\ref{thm:partial-balayage-unbdd-sol-fn-props}, the assumptions
on the signed measure $\sigma$ for this partial balayage measure to
exist do not need to be very harsh---negative total mass and (for
instance) continuity of the potential of the negative part of $\sigma$
are sufficient. It thus seems rather likely that there exists a limit
of the generalized divisible sandpile model for $d = 2$ as the
confining radius grows infinitely large. We have unfortunately been
unable to prove such a result, and will have to settle with a
conjecture:
\begin{conjecture}
  Let $\sigma$ be a generalized mass configuration on $\xi \Z^2$ with
  finite support. Then $M_R \to 0$ as $R \to \infty$, where $M_R$ is
  the boundary mass of $\gds_R(\sigma)$:
  \begin{align*}
    M_R := \sum_{x \in \xi \Z^2} (\gds_R(\sigma))_+(x).
  \end{align*}
\end{conjecture}
As a final remark, we note that one way of interpreting such a
result---if it holds---is that the confining radius $R > 0$ that we
introduced to ensure convergence of the generalized model in a sense
is unnecessary in dimension $d = 2$. On the other hand, based on the
above example the confining radius seems required in dimensions
$d \geq 3$. Given the recently developed strong connections between
the standard divisible sandpile and the so-called internal diffusion
limited aggregation (IDLA) model for particle aggregation, which uses
simple random walks as a means to relocate excess mass, it does not
seem too unlikely that the apparent difference in behaviour between
$d = 2$ and $d \geq 3$ for the generalized divisible sandpile may have
something to do with the result by G.~P{\'o}lya~\cite{polya-21} that
the simple random walk is recurrent in dimension $d = 2$ and transient
if $d \geq 3$.

\bibliography{refs}

\newpage
\noindent Joakim Roos \\[.1cm]
Department of Mathematics \\
KTH \\
SE-100 44 Stockholm \\
Sweden \\[.1cm]
e-mail: \texttt{joakimrs@math.kth.se}

\end{document}